\documentclass{amsart}

\usepackage{amssymb,amsmath,amsthm,amsfonts}
\usepackage{framed,comment,enumerate}
\usepackage[pdftex]{graphicx}
\usepackage{xcolor, framed}
\usepackage{tikz}

\usepackage{verbatim}

\usepackage[showonlyrefs]{mathtools}

\usepackage[sort,nocompress]{cite}

\def\R {\mathbb{R}}
\def\CC {\mathbb{C}}

\def\N {\mathbb{N}}

\def\eps{\varepsilon}
\def\dist{{\rm dist}}

\newcommand{\wc}{\rightharpoonup}

\newcommand{\pa}{\partial}
\newcommand{\mf}[1]{\mathbf{#1}}

\DeclareMathOperator{\rad}{rad}

\newtheorem{proposition}{Proposition}[section]
\newtheorem{theorem}[proposition]{Theorem}
\newtheorem*{theorem*}{Theorem}

\newtheorem{lemma}[proposition]{Lemma}

\theoremstyle{definition}
\newtheorem{remark}[proposition]{Remark}
\numberwithin{equation}{section}

\title[Normalized solutions]{Normalized solutions for a system of coupled cubic Schr\"odinger equations on $\R^3$}

\author{Thomas Bartsch, Louis Jeanjean and Nicola Soave}

\address{
\hbox{\parbox{5.7in}{\medskip\noindent
Thomas Bartsch\\
Mathematisches Institut, Justus-Liebig-Universit\"at Giessen, \\
Arndtstrasse 2, 35392 Giessen (Germany),\\[2pt]
{\em{E-mail address: }}{\tt Thomas.Bartsch@math.uni-giessen.de.} \\ [5pt]
Louis Jeanjean\\
Laboratoire de Math\'ematiques (UMR 6623), Universit\'e de Franche-Comt\'e, \\
16, Route de Gray, 25030 Besan\c{c}on Cedex (France),\\
{\em{E-mail address: }}{\tt louis.jeanjean@univ-fcomte.fr.}\\[5pt]
Nicola Soave\\
Mathematisches Institut, Justus-Liebig-Universit\"at Giessen, \\
Arndtstrasse 2, 35392 Giessen (Germany),\\[2pt]
{\em{E-mail address: }}{\tt nicola.soave@gmail.com, Nicola.Soave@math.uni-giessen.de.}}}}

\thanks{{\em Acknowledgements}. Nicola Soave is partially supported through the project ERC Advanced Grant 2013 n. 339958 ``Complex Patterns for Strongly Interacting Dynamical Systems - COMPAT''. This work has been carried out in the framework of the project NONLOCAL (ANR-14-CE25-0013), funded by the French National Research Agency (ANR)}


\begin{document}

\begin{abstract}
We consider the system of coupled elliptic equations
\[
\begin{cases}
-\Delta u - \lambda_1 u = \mu_1 u^3+ \beta u v^2 \\
-\Delta v- \lambda_2 v = \mu_2 v^3 +\beta u^2 v
\end{cases} \text{in $\R^3$},
\]
and study the existence of positive solutions satisfying the additional condition
\[
\int_{\R^3} u^2 = a_1^2 \quad \text{and} \quad \int_{\R^3} v^2 = a_2^2.
\]
Assuming that $a_1,a_2,\mu_1,\mu_2$ are positive fixed quantities, we prove existence results for different ranges of the coupling parameter $\beta>0$. The extension to systems with an arbitrary number of components is discussed, as well as the orbital stability of the corresponding standing waves for the related Schr\"odinger systems.
\end{abstract}

\maketitle

\section{Introduction}

This paper concerns the existence of solutions $(\lambda_1,\lambda_2,u,v) \in \R^2 \times H^1(\R^3,\R^2)$ to the system of elliptic equations
\begin{equation}\label{system}
\begin{cases}
-\Delta u - \lambda_1 u = \mu_1 u^3+ \beta u v^2 \\
-\Delta v- \lambda_2 v = \mu_2 v^3 +\beta u^2 v
\end{cases} \text{in $\R^3$},
\end{equation}
satisfying the additional condition
\begin{equation}\label{normalization}
\int_{\R^3} u^2 = a_1^2 \quad \text{and} \quad \int_{\R^3} v^2 = a_2^2.
\end{equation}
One refers to this type of solutions as to \emph{normalized solutions}, since \eqref{normalization} imposes a normalization on the $L^2$-masses of $u$ and $v$. This fact implies that $\lambda_1$ and $\lambda_2$ cannot be determined a priori, but are part of the unknown.

The problem under investigation comes from the research of solitary waves for the system of coupled Schr\"odinger equations
\begin{equation}\label{syst schrod}
\begin{cases}
- \iota \pa_t \Phi_1 = \Delta \Phi_1 + \mu_1 |\Phi_1|^2 \Phi_1+ \beta |\Phi_2|^2 \Phi_1 \\
- \iota \pa_t \Phi_2 = \Delta \Phi_2 + \mu_2 |\Phi_2|^2 \Phi_2+ \beta |\Phi_1|^2 \Phi_2
\end{cases} \text{in $\R \times \R^3$},
\end{equation}
having applications in nonlinear optics and in the Hartree-Fock approximation for Bose-Einstein condensates with multiple states; see \cite{esry-etal:1997,malomed:2008}.

It is well known that three quantities are conserved in time along trajectories of
\eqref{syst schrod}: the \emph{energy}
\begin{multline*}
J_{\CC}(\Phi_1,\Phi_2)= \int_{\R^3} \left(\frac{1}{2}|\nabla \Phi_1|^2 -\frac{\mu_1}{4} |\Phi_1|^4\right) \\
+ \int_{\R^3} \left(\frac{1}{2}|\nabla \Phi_2|^2 -\frac{\mu_2}{4} |\Phi_2|^4\right) - \frac{\beta}{2} \int_{\R^3} |\Phi_1|^2 |\Phi_2|^2,
\end{multline*}
and the \emph{masses}
\[
\int_{\R^3} |\Phi_1|^2 \quad \text{and} \quad \int_{\R^3} |\Phi_2|^2.
\]
A solitary wave of \eqref{syst schrod} is a solution having the form
\[
\Phi_1(t,x) = e^{-i\lambda_1 t} u(x) \quad \text{and} \quad \Phi_2(t,x) = e^{-i\lambda_2 t} v(x)
\]
for some $\lambda_1,\lambda_2 \in \R$, where $(u,v)$ solves \eqref{system}. Two different approaches are possible: one can either regard the frequencies $\lambda_1,\lambda_2$ as fixed, or include them in the unknown and prescribe the masses. In this latter case, which seems to be particularly interesting from the physical point of view, $\lambda_1$ and $\lambda_2$ appear as Lagrange multipliers with respect to the mass constraint.

The problem with fixed $\lambda_i$ has been widely investigated in the last ten years, and, at least for systems with $2$ components and existence of positive solutions (i.~e.\ $u,v>0$ in $\R^3$), the situation is quite well understood. A complete review of the available results in this context goes beyond the aim of this paper; we refer the interested reader to \cite{AmbrosettiColorado, Bartsch, BaDaWa, BaWa, BaWaWei, ChZo, LinWei, LinWei2, LiuWang, MaiaMontefuscoPellacci, Mand, SaWa, Sirakov, Soave, SoaveTavares, TerVer, WeiWeth} and to the references therein.

In striking contrast, very few papers deal with the existence of normalized solutions. Up to our knowledge, the only known results are the ones in \cite{BarJea,NoTaTeVeJEMS,NoTaVe2,TaTe}. In \cite{NoTaVe2}, the authors consider \eqref{system} in bounded domains of $\R^N$, or the problem with trapping potentials in the whole space $\R^N$ (the presence of a trapping potential makes the two problems essentially equivalent), with $N \le 3$. In both cases, they proved existence of positive solutions with small masses $a_1$ and $a_2$, and the orbital stability of the associated solitary waves, see Theorem~1.3 therein. It is remarkable that they can work essentially without assumptions on $\mu_1,\mu_2$ and $\beta$. The requirement that the masses have to be small gives their result a bifurcation flavor. In \cite{NoTaTeVeJEMS,TaTe} the authors consider the defocusing setting $\mu_1,\mu_2 <0$ in regime of competition $\beta<0$ in bounded domains. In the defocusing competitive case $\mu_1,\mu_2,\beta<0$ existence is an easy consequence of standard Lusternik-Schnirelmann theory because the functional is bounded from below. Supposing that all the components have the same mass, they prove existence of infinitely many solutions and occurrence of phase-separation as $\beta \to -\infty$. Concerning \cite{BarJea}, we postpone a discussion of the results therein in the following paragraphs.

In the present paper we address a situation which is substantially different compared to those considered in the papers \cite{NoTaTeVeJEMS,NoTaVe2,TaTe}. We study system \eqref{system} in $\R^3$ in the focusing setting $\mu_1,\mu_2>0$, so that the functional is unbounded from below on the constraint. We prove the existence of positive normalized solutions for different ranges of the coupling parameter $\beta>0$, without any assumption on the masses $a_1,a_2$. Our approach is variational: we find solutions of \eqref{system}-\eqref{normalization} as critical points of the energy functional
\begin{equation}\label{def energy}
J(u,v) = \int_{\R^3} \left(\frac12|\nabla u|^2 -\frac{\mu_1}{4} u^4\right)
          + \int_{\R^3} \left(\frac12|\nabla v|^2 -\frac{\mu_2}{4} v^4\right)
          - \frac{\beta}{2} \int_{\R^3} u^2 v^2,
\end{equation}
on the constraint $T_{a_1} \times T_{a_2}$, where for $a \in \R$ we define
\begin{equation}\label{def sphere}
T_a:=\left\{ u \in H^1(\R^3): \int_{\R^3} u^2 = a^2 \right\}.
\end{equation}
The main results are the following:

\begin{theorem}\label{thm: existence for beta small}
Let $a_1,a_2,\mu_1,\mu_2>0$ be fixed. There exists $\beta_1>0$ depending on $a_i$ and $\mu_i$ such that if $0<\beta < \beta_1$, then \eqref{system}-\eqref{normalization} has a solution $(\tilde \lambda_1,\tilde \lambda_2,\tilde u,\tilde v)$ such that $\tilde \lambda_1, \tilde \lambda_2<0$, and $\tilde u$ and $\tilde v$ are both positive and radial.
\end{theorem}

For our next result we introduce a Pohozaev-type constraint as follows:
\begin{equation}\label{def constainA}
V:= \left\{ (u,v) \in T_{a_1} \times T_{a_2} : G(u,v) = 0 \right\},
\end{equation}
where
\[
G(u,v) = \int_{\R^3} \left( |\nabla u|^2 + |\nabla v|^2\right) - \frac{3}{4} \int_{\R^3} \left( \mu_1 u^4 + 2\beta u^2 v^2 +\mu_2 v^4 \right).
\]
We shall see that $V$ contains all solutions of \eqref{system}-\eqref{normalization}.

\begin{theorem}\label{thm: existence for beta large}
Let $a_1,a_2,\mu_1,\mu_2>0$ be fixed. There exists $\beta_2>0$ depending on $a_i$ and $\mu_i$ such that, if $\beta > \beta_2$, then \eqref{system}-\eqref{normalization} has a solution $(\bar \lambda_1,\bar \lambda_2,\bar u,\bar v)$ such that $\bar \lambda_1, \bar \lambda_2<0$, and $\bar u$ and $\bar v$ are both positive and radial. Moreover, $(\bar \lambda_1,\bar \lambda_2,\bar u,\bar v)$ is a ground state solution in the sense that
\[
\begin{aligned}
J(\bar u,\bar v)
 &= \inf\{J(u,v): (u,v)\in V\}\\
 &= \inf\left\{J(u,v):\text{$(u,v)$ is a solution of \eqref{system}-\eqref{normalization} for
     some $\lambda_1,\lambda_2$} \right\}
\end{aligned}
\]
holds.
\end{theorem}

\begin{remark} a) The values of $\beta_1$ and $\beta_2$ can be explicitly estimated; see \eqref{def_beta_1} and Remark \ref{rem: estimate beta_2} below. In particular, we point out that they are not obtained by means of any limit process, so that one should not think that $\beta_1$ is very small and $\beta_2$ is very large. For instance, if $\mu_1=\mu_2$ and $a_1\le a_2$, then the proof of Theorem \ref{thm: existence for beta small} works for
\[
\beta_1=\mu_1\left(\sqrt{1+\frac{a_1^2}{a_2^2}}-1\right).
\]
Nevertheless, it remains an open problem to obtain sharp bounds.

b) The variational characterizations of the solutions obtained in Theorems 1.1 and 1.2 are different. The solution from Theorem~\ref{thm: existence for beta small} has Morse index 2 as critical point of $J$ constrained to $T_{a_1}\times T_{a_2}$. On the other hand, the solution from Theorem~\ref{thm: existence for beta large} is a mountain pass solution of $J$ on the constraint. It can also be obtained as a minimizer of $J$ on the Pohozaev-type submanifold of the constraint.

c) Our results can be extended with minor changes to systems with general exponents of type
\begin{equation}\label{power-type sys}
\begin{cases}
-\Delta u_1 -\lambda_1 u_1 = \mu_1 |u_1|^{2p_1-2} u_1 +  \beta |u_1|^{r-2} |u_2|^{r} u_1 \\
-\Delta u_2 -\lambda_2 u_2 = \mu_2 |u_2|^{2p_2-2} u_2 + \beta |u_1|^{r} |u_2|^{r-2} u_2
\end{cases} \qquad \text{in $\R^N$}
\end{equation}
(or the $k$ components analogue) with $N \le 4$, provided we restrict ourselves to a \emph{$L^2$-supercritical} and \emph{Sobolev subcritical} setting:
\[
2+\frac{4}{N}< 2p_i,2r < \frac{2N}{N-2}.
\]
Moreover, the proofs do not use the evenness of the functional. Thus one may replace the terms $u^4$, $v^4$ in \eqref{def energy} by general nonlinearities $f(u),g(v)$ which are not odd. Similarly the coupling term $u^2v^2$ in the functional may be replaced by a nonsymmetric one. We decided not to include this kind of generality since it would make the statement of our results and the proofs very technical.

d) Also in the case of fixed frequencies for system \eqref{system} there exist values $0<\beta_1'<\beta_2'$ such that the problem has a positive solution whenever $\beta<\beta_1'$ or $\beta>\beta_2'$ \cite{AmbrosettiColorado,  Sirakov}, see also \cite{MaiaMontefuscoPellacci}. In this setting, it is known that if $\lambda_1 \ge \lambda_2$, $\mu_1 \ge \mu_2$, and one of the inequalities is strict, then $\beta_1'<\beta_2'$, and for $\beta \in [\beta_1',\beta_2']$ the problem has no positive solution \cite{BaWa, Sirakov}. On the other hand, the non-existence range (in terms of $\beta$) can disappear. This is the case, for instance, if $\lambda_1=\lambda_2=\mu_1=\mu_2=1$. Then \eqref{system} has positive solutions for all $\beta>0$. Since in the context of normalized solutions the values $\lambda_i$ are not prescribed, it is an interesting open problem whether there are conditions on $a_1,a_2,\mu_1,\mu_2$ such that positive solutions of \eqref{system}-\eqref{normalization} exist for all $\beta>0$.

e) Despite the similarity between our results and those in \cite{AmbrosettiColorado, Sirakov}, the proofs differ substantially. First, while in  \cite{AmbrosettiColorado, Sirakov} the approach is based on the research of critical points constrained on Nehari-type sets associated to the problem, here no Nehari manifold is available, since $\lambda_1$ and $\lambda_2$ are part of the unknown; as a consequence, we shall directly investigate the geometry of the functional on the product of the $L^2$-spheres $T_{a_1} \times T_{a_2}$ in order to apply a suitable minimax theorem. We also point out that in \cite{AmbrosettiColorado, Sirakov}, as well as in all the contributions related to the problem with fixed frequencies, one of the main difficulties is represented by the fact that one search for solutions having both $u \not \equiv 0$ and $v \not \equiv 0$. Here this problem is still present, and actually it assumes a more subtle form, in the following sense: let us suppose that we can find a Palais-Smale sequence for $J$ on $T_{a_1} \times T_{a_2}$, and suppose that this sequence is weakly convergent in $H^1$ to a limit $(u,v)$. Due to the lack of compactness of the embedding $H^1(\R^3) \hookrightarrow L^2(\R^3)$, a delicate step consists in showing that $(u,v) \in T_{a_1} \times T_{a_2}$, so that it satisfies \eqref{normalization}. Notice that the lack of compactness persists also if we restrict ourselves to a radial setting. As a consequence, we emphasize that it is not sufficient to rule out the possibility that in the weak limit $u \equiv 0$ or $v \equiv 0$. We have also to prevent the loss of part of the mass of one of the components in the passage to the limit.  
\end{remark}

Both theorems rest upon a suitable minimax argument, where an important role is played by the \emph{ground state levels} $\ell(a_1,\mu_1)$ and $\ell(a_2,\mu_2)$ associated to the scalar problems
\[
\begin{cases}
-\Delta w - \lambda w= \mu w^3 & \text{in $\R^3$}  \\
\int_{\R^3} w^2 = a^2
\end{cases}
\]
with $a=a_1$ and $\mu=\mu_1$, or with $a=a_2$ and $\mu=\mu_2$, respectively.
We refer to Section \ref{sec: preliminaries} for the precise definition of $\ell(a,\mu)$. In this perspective, it is interesting to emphasize the different relations between the critical values of Theorems \ref{thm: existence for beta small} and \ref{thm: existence for beta large} with $\ell(a_1,\mu_1)$ and $\ell(a_2,\mu_2)$.

\begin{proposition}
With the notation of Theorems \ref{thm: existence for beta small} and \ref{thm: existence for beta large}, we have
\[
J(\bar u, \bar v) < \min\{\ell(a_1,\mu_1), \ell(a_2,\mu_2)\} \le \max \{\ell(a_1,\mu_1), \ell(a_2,\mu_2)\} < J(\tilde u,\tilde v).
\]
\end{proposition}

In \cite{BarJea} the authors consider systems of the type of \eqref{power-type sys} looking also for solutions satisfying \eqref{normalization}. The results obtained in \cite{BarJea} have no intersection with the one of the present paper because there $2<p_1<2+4/N<p_2<6$. A common feature is that one looks for constrained critical points in a situation where the functional is unbounded from below on the constraint. Already in the scalar case it is known that, when the underlying equations are set on all the space, looking to critical points which are not global minima of the associated functional may present new difficulties (with respect to the minimizing problem), see \cite{Jeanjean, BeJeLu}. In particular a standard approach following the Compactness Concentration Principle of  P.L. Lions \cite{Li1,Li2} is hardly applicable. We also mention \cite{BaVa,JeLuWa,Lu} for multiplicity results in that direction, and \cite{NoTaVe1} for normalized solutions in bounded domains.

In the second part of the paper we partially generalize the previous results to the $k\ge 2$ components system
\begin{equation}\label{system many}
\begin{cases}
-\Delta u_i -\lambda_i u_i = \sum_{j=1}^k \beta_{ij} u_j^2 u_i & \text{in $\R^3$} \\
u_i \in H^1(\R^3)
\end{cases} \qquad i=1,\dots,k,
\end{equation}
with the normalization condition
\begin{equation}\label{norm many}
\int_{\R^3} u_i^2 = a_i^2 \qquad i=1,\dots,k.
\end{equation}
We always suppose that $\beta_{ij}=\beta_{ji}$ for every $i \neq j$. Notice that problem \eqref{system}-\eqref{normalization} falls in this setting with $k=2$, $u=u_1$, $v=u_2$, $\beta_{ii}=\mu_i$ and $\beta_{12}=\beta$.

From a variational point of view, thanks to the fact that $\beta_{ij}=\beta_{ji}$ solutions of \eqref{system many}-\eqref{norm many} are critical points of
\[
J(u_1,\dots,u_k):=  \int_{\R^3} \left(\frac{1}{2} \sum_{i=1}^k|\nabla u_i|^2 -\frac{1}{4}\sum_{i,j=1}^k \beta_{ij} u_i^2 u_j^2\right)
\]
on the constraint $T_{a_1} \times \dots \times T_{a_k}$, where $T_a$ has been defined in \eqref{def sphere}. Notice that the definition of the functional $J$ depends on $k$ and the matrix $\beta_{ij})$, but we will not stress such dependence to keep the notation as simple as possible.

The first result which we present is the extension of Theorem \ref{thm: existence for beta large} to any $k \ge 3$.

\begin{theorem}\label{thm: large couplings many comp}
Let $k \ge 2$, and let $a_i,\beta_{ii},\beta_{ij}>0$ be positive constant, such that the following inequality holds:
\begin{equation}\label{condition couplings}
\frac{\displaystyle \left( \sum_{i=1}^k a_i^2 \right)^3}{ \displaystyle \left( \sum_{i,j=1}^k \beta_{ij} a_i^2 a_j^2\right)^2 } < \min_{\substack{ \mathcal{I} \subset \{1,\dots,k\} \\ |\mathcal{I}| \le k-1 }}\frac{1}{ \displaystyle \left[ \max_{i \in \mathcal{I}} \{\beta_{jj} a_j\} + \frac{k-2}{k-1}\max_{\substack{i \neq j  \\ i,j \in \mathcal{I} }} \left\{  \beta_{ij} a_i^{1/2} a_j^{1/2}\right\}\right]^2},
\end{equation}
where $|\mathcal{I}|$ denotes the cardinality of the set $\mathcal{I}$. Then \eqref{system many}-\eqref{norm many} has a solution $(\bar \lambda_1,\dots,\bar \lambda_k,\bar u_1,\dots,\bar u_k)$ such that $\bar \lambda_i<0$, and $\bar u_i$ is positive and radial for every $i$. Moreover,
\[
J(\bar u_1,\dots,\bar u_k) = \inf\left\{J( u_1,\dots,u_k): \text{$(u_1,\dots,u_k)$ is a solution of \eqref{system many}-\eqref{norm many}} \right\},
\]
that is $(\bar \lambda_1,\dots,\bar \lambda_k,\bar u_1,\dots,\bar u_k)$ is a ground state solution.
\end{theorem}

Some remarks are in order. 

\begin{remark}\label{rem: on multi-comp} a) The set of parameters fulfilling condition \eqref{condition couplings} is not empty. For instance, if $a_i=a$ for every $i$, $\beta_{ii}>0$ are fixed and $\beta_{ij}=\beta$ for every $i \neq j$, then \eqref{condition couplings} is satisfied provided $\beta$ is sufficiently large. More in general, if $\beta_{ii}>0$,  $\beta_{ij}=\beta$ for every $i \neq j$, and
\[
\frac{\left(\sum_i a_i^2 \right)^3 \left(\frac{k-2}{k-1}\right)^2 \max_{i \neq j} \{a_i a_j\} }{\left( \sum_{i \neq j} a_i^2 a_j^2\right)^2} <1,
\]
then \eqref{condition couplings} is satisfied provided $\beta$ is sufficiently large.

b) At a first glance \eqref{condition couplings} seems unclear if compared with the simple condition $\beta > \beta_2$ appearing in Theorem \ref{thm: existence for beta large}. On the contrary, for $\beta_{ii}$ and $a_i$ fixed and $k=2$, it is easy to check that \eqref{condition couplings} is fulfilled provided $\beta_{12}$ is larger than a positive threshold $\beta_2'$ (which can be explicitly computed). We observed that the value $\beta_2$ in Theorem \ref{thm: existence for beta large} can be estimated, see Remark \ref{rem: estimate beta_2}. Actually, using \eqref{condition couplings} we expect to have a better estimate (in the sense that $\beta_2' \le \beta_2$); the price to pay is that the derivation of \eqref{condition couplings} requires a lot of extra work.

c) A condition somehow similar to \eqref{condition couplings} appears also for the problem with fixed frequencies $\lambda_i$, see Theorem 2.1 in \cite{LiuWang}.
\end{remark}

Regarding the extension of Theorem \ref{thm: existence for beta small} to systems with an arbitrary number of components, we have a much weaker result.

\begin{proposition}\label{prop: small couplings}
Let $a_i,\beta_{ii}>0$ be fixed positive constant. There exists $\beta_0>0$ such that if $|\beta_{ij}|<\beta_0$ for every $i \neq j$, then system \eqref{system many}-\eqref{norm many} has a solution $(\tilde \lambda_1,\dots,\tilde \lambda_k,\tilde u_1,\dots,\tilde u_k)$ such that $\lambda_i<0$, and $u_i$ is positive and radial for every $i$.
\end{proposition}

The proof is based on a simple application of the implicit function theorem, and is omitted for the sake of brevity. Notice that using a perturbative argument we can allow some (or all) the couplings $\beta_{ij}$ to take negative values. On the other hand, being $\beta_0$ obtained by a limit argument, it cannot be estimated from below and it could be very small; in this sense Proposition \ref{prop: small couplings} is weaker than Theorem \ref{thm: existence for beta small}, where an explicit estimate for $\beta_1$ is available.

\medskip

Let us now turn to the question of the orbital stability of the solitary waves of
\begin{equation}\label{syst schrod many}
- \iota \pa_t \Phi_j = \Delta \Phi_j + \beta_{jj} |\Phi_j|^2 \Phi_j+ \sum_{k \neq j}\beta_{kj} |\Phi_k|^2 \Phi_j \qquad \text{in $\R \times \R^3$, $j=1,\dots,k$},
\end{equation}
associated to the solutions found in Theorem \ref{thm: large couplings many comp} (or Theorem \ref{thm: existence for beta large} if $k=2$). In this framework, we can adapt the classical Berestycki-Cazenave argument \cite{BerCaz} (see also \cite{Caz,LeCoz} for more detailed proofs) and prove the following:

\begin{theorem}\label{thm: instability}
Let $k \ge 2$, and $(\bar \lambda_1,\dots,\bar \lambda_k,\bar u_1,\dots,\bar u_k)$ be the solution obtained in Theorem \ref{thm: large couplings many comp} (or in Theorem \ref{thm: existence for beta large} if $k = 2$). Then the associated  solitary wave is orbitally unstable.
\end{theorem}

Regarding the stability of the solutions found in Theorem \ref{thm: existence for beta small} and Proposition \ref{prop: small couplings}, a Berestycki-Cazenave-type argument does not seem to be applicable, since these solutions are characterized by a different minimax construction with respect to those in Theorems \ref{thm: existence for beta large} and \ref{thm: large couplings many comp}. Therefore, the stability remains open in these cases.

\medskip

The orbital stability of solutions to weakly coupled Schr\"odinger equations associated to power-type systems like \eqref{power-type sys} has been studied in several papers (we refer to \cite{Cor,MaMoPe2, NgWa, Ohta} and to the references therein), but the available results mainly regard the $L^2$-subcritical setting setting $2p<1+4/N$, and the problem with fixed frequencies. In particular, we point out that Theorem \ref{thm: instability} does not follow by previous contributions.

\section{Preliminaries}\label{sec: preliminaries}

In the first part of the section, we collect some facts concerning the cubic NLS equation, which will be used later. Let us consider the scalar problem
\begin{equation}\label{scalar pb}
\begin{cases}
-\Delta w + w= w^3 & \text{in $\R^3$} \\
w>0    & \text{in $\R^3$} \\
w(0) = \max w \quad \mbox{and} \quad  w \in H^1(\R^3).
\end{cases}
\end{equation}
It is well known that \eqref{scalar pb} has a unique solution, denoted by $w_0$ and that this solution is radial. In what follows we set
\begin{equation}\label{def of constants}
C_0:= \int_{\R^3} w_0^2 \quad \text{and} \quad C_1:= \int_{\R^3} w_0^4.
\end{equation}

For $a,\mu \in \R$ fixed, let us search for $(\lambda,w) \in \R \times H^1(\R^3)$, with $\lambda<0$ in $\R^3$, solving
\begin{equation}\label{scalar pb with mass}
\begin{cases}
-\Delta w - \lambda w= \mu w^3 & \text{in $\R^3$}  \\
w(0) = \max w \quad \mbox{ and }  \quad \int_{\R^3} w^2 = a^2.
\end{cases}
\end{equation}
Solutions $w$ of \eqref{scalar pb with mass} can be found as critical points of $I_\mu: H^1(\R^3) \mapsto \R$, defined by
\begin{equation}\label{scalar functional}
I_\mu(w) = \int_{\R^3} \left(\frac{1}{2}|\nabla w|^2 -\frac{\mu}{4} w^4\right),
\end{equation}
constrained on the $L^2$-sphere $T_a$, and $\lambda$ appears as Lagrange multipliers. It is well known that they can be obtained by the solutions of \eqref{scalar pb} by scaling.

Let us introduce the set
\begin{equation}\label{def of P}
\mathcal{P}(a,\mu):= \left\{ w \in T_a: \int_{\R^3} |\nabla w|^2 = \frac{3 \mu }{4} \int_{\R^3} w^4 \right\}.
\end{equation}
The role of $\mathcal{P}(a,\mu)$ is clarified by the following result.

\begin{lemma}\label{lem: natural constraint}
If $w$ is a solution of \eqref{scalar pb with mass}, then $w \in \mathcal{P}(a,\mu)$. In addition the positive solution $w$ of \eqref{scalar pb with mass} minimizes $I_{\mu}$ on $\mathcal{P}(a,\mu)$.
\end{lemma}
\begin{proof}
The proof of the first part is a simple consequence of the Pohozaev identity. We refer to Lemma 2.7 in \cite{Jeanjean} for more details. For the last part we refer to Lemma 2.10 in \cite{Jeanjean}.
\end{proof}

\begin{proposition}\label{prop: explicit}
Problem \eqref{scalar pb with mass} has a unique positive solution $(\lambda_{a,\mu}, w_{a,\mu})$ defined by
\[
\lambda_{a,\mu} := -\frac{C_0^2}{\mu^2 a^4} \quad \text{and} \quad
w_{a,\mu}(x):= \frac{C_0 }{\mu^{3/2} a^2} w_0 \left( \frac{C_0}{\mu a^2} x \right).
\]
The function $w_{a,\mu}$ satisfies
\begin{align}
\int_{\R^3} |\nabla w_{a,\mu}|^2 &= \frac{3C_0 C_1}{4\mu^2 a^2}  \label{int grad}\\
\int_{\R^3} w_{a,\mu}^4 & = \frac{C_0 C_1}{\mu^3 a^2}. \label{int fourth} \\
\ell (a,\mu):= I_\mu(w_{a,\mu}) & = \frac{C_0 C_1 }{8 \mu^2 a^2}.  \label{espressione funionale}
\end{align}
The value $\ell(a,\mu)$ is called \emph{least energy level} of problem \eqref{scalar pb with mass}.
\end{proposition}

\begin{proof}
It is not difficult to directly check that $w_{a,\mu}$ defined in the proposition is a solution of \eqref{scalar pb with mass} for $\lambda= \lambda_{a,\mu}<0$. By \cite{Kwong}, it is the only positive solution.
To obtain \eqref{int grad} and \eqref{int fourth}, we can use the explicit expression of $w_{a,\mu}$: by a change of variables
\[
\int_{\R^3} |\nabla w_{a,\mu}|^2 = \frac{C_0}{\mu^2 a^2} \int_{\R^3} |\nabla w_0|^2 = \frac{3C_0 }{4\mu^2 a^2} \int_{\R^3} w_0^4,
\]
where the last equality follows by Lemma \ref{lem: natural constraint} with $a^2=C_0$ and $\mu = 1$. This gives \eqref{int grad}. In a similar way, one can also prove \eqref{int fourth} and \eqref{espressione funionale}.
\end{proof}

Working with systems with several components, it will be useful to have a characterization of the best constant in a Gagliardo-Nirenberg inequality in terms of $C_0$ and $C_1$. To obtain it, we observe at first that if $w_a:= w_{a,C_0/a^2}$, then $w_a$ is the unique positive solution of
\[
\begin{cases}
-\Delta w + w= \frac{C_0}{a^2} w^3 & \text{in $\R^3$}  \\
w(0) = \max w \quad \mbox{and} \quad  \int_{\R^3} w^2 = a^2,
\end{cases}
\]
and hence is a minimizer of $I_{a,C_0/a^2}$ on $\mathcal{P}(a,C_0/a^2)$. Our next result shows that this level can also be characterized as an infimum of a Rayleigh-type quotient, defined by
\[
\mathcal{R}_a(w):= \frac{8 \left( \int_{\R^3} |\nabla w|^2\right)^3}{27\left(\frac{C_0}{a^2} \int_{\R^3} w^4\right)^2}.
\]
\begin{lemma}\label{leas esnergy as ray 1}
There holds
\[
\inf_{\mathcal{P}(a,C_0/a^2)} I_{a,C_0/a^2} = \inf_{T_a} \mathcal{R}_a.
\]
\end{lemma}
\begin{proof}
We refer to the proof of the forthcoming Lemma \ref{lem: rayleigh}, where the corresponding equality is proved for systems, and which then includes the present result as a particular case.
\end{proof}

Let us recall the following Gagliardo-Nirenberg inequality: there exists a universal constant $S>0$ such that
\begin{equation}\label{Gagliardo-Nirenberg}
\int_{\R^3} w^4 \le S \left(\int_{\R^3} w^2\right)^{1/2} \left( \int_{\R^3} |\nabla w|^2 \right)^{3/2}  \qquad \text{for all }w \in H^1(\R^3).
\end{equation}
In particular, the optimal value of $S$ can be found as
\begin{equation}\label{relazione 102}
\frac{1}{S^2} = \inf_{w \in H^1(\R^3) \setminus \{0\}} \frac{ \left(\int_{\R^3} w^2 \right) \cdot \left( \int_{\R^3} |\nabla w|^2  \right)^{3} }{\left(\int_{\R^3} w^4\right)^2} = \inf_{w \in T_a} \frac{  a^2 \left( \int_{\R^3} |\nabla w|^2  \right)^{3}}{\left(\int_{\R^3} w^4\right)^2},
\end{equation}
where the last equality comes from the fact that the ratio on the right hand side is invariant with respect to multiplication of $w$ with a positive number.

\begin{lemma}\label{lem: value S}
In the previous notation, we have
\[
S^2 = \frac{64}{27C_0 C_1},
\]
where $C_0$ and $C_1$ have been defined in \eqref{def of constants}.
\end{lemma}
\begin{proof}
Multiplying and dividing the last term in \eqref{relazione 102} by $8a^2 /(27 C_0^2)$, we deduce that
\[
\frac{1}{S^2} = \frac{27C_0^2}{8 a^2}   \inf_{w \in T_a} \mathcal{R}_a(w).
\]
Hence, by Proposition \ref{prop: explicit} and Lemma \ref{leas esnergy as ray 1}, we infer that
\[
\frac{1}{S^2}= \frac{27C_0^2}{8 a^2}  I_{C_0/a^2}(w_{a,C_0/a^2})  = \frac{27C_0 C_1 }{64}. \qedhere
\]
\end{proof}

\section{Proof of Theorem \ref{thm: existence for beta small}}

This section is devoted to the proof of Theorem \ref{thm: existence for beta small}, which is based upon a two-dimensional linking argument.

In order to avoid compactness issues, we work in a radial setting. This means that we search for solutions of \eqref{system}-\eqref{normalization} as critical points of $J$ constrained on $S_{a_1} \times S_{a_2}$, where for any $a \in \R$ the set $S_a$ is defined by
\begin{equation}\label{sphere radial}
S_a:= \left\{ w \in H^1_{\rad}(\R^3): \int_{\R^3} w^2= a^2 \right\},
\end{equation}
and $H^1_{\rad}(\R^3)$ denotes the subset of $H^1(\R^3)$ containing all the functions which are radial with respect to the origin. Recall that $H^1_{\rad}(\R^3) \hookrightarrow L^4(\R^3)$ with compact embedding, and the fact that critical points of $J$ constrained on $S_{a_1} \times S_{a_2}$ (thus in a radial setting) are true critical points of $J$ constrained in the full product $T_{a_1} \times T_{a_2}$ is a consequence of the Palais' principle of symmetric criticality.

\medskip

In order to describe the minimax structure, it is convenient to introduce some notation. We define, for $s \in \R$ and $w \in H^1(\R^3)$, the radial dilation
\begin{equation}\label{def star}
(s \star w)(x) := e^{\frac{3s}{2}} w(e^s x).
\end{equation}
It is straightforward to check that if $w \in S_a$, then $(s \star w) \in S_a$ for every $s\in\R$.

\begin{lemma}\label{lem: basic properties radial dilation}
For every $\mu >0$ and $w \in H^1(\R^3)$, there holds:
\begin{align*}
I_\mu(s \star w) &= \frac{e^{2s}}{2} \int_{\R^3} |\nabla w|^2 - \frac{e^{3s}}{4}\mu \int_{\R^3} w^4 \\
\frac{\pa}{\pa s} I_\mu(s \star w) &= e^{2s} \left( \int_{\R^3} |\nabla w|^2 - \frac{3 e^{s}}{4} \mu\int_{\R^3} w^4\right).
\end{align*}
In particular, if $w=w_{a,\mu}$, then
\[
\frac{\pa}{\pa s} I_\mu(s \star w_{a,\mu}) \text{ is } \begin{cases} >0 & \text{if $s<0$} \\ = 0& \text{if $s=0$} \\ <0 & \text{if $s>0$}. \end{cases}
\]
\end{lemma}
For the reader's convenience, we recall that $I_\mu$ denotes the functional for the scalar equation, see \eqref{scalar functional}, and $w_{a,\mu}$ has been defined in Proposition \ref{prop: explicit}.
\begin{proof}
For the first part, it is sufficient to use the definition of $s \star w$ and a change of variables in the integrals. For the second part, we observe that
\[
\frac{\pa}{\pa s} I_\mu(s \star w_{a,\mu}) \text{ is } \begin{cases} >0 & \text{if $s<\bar s$} \\ = 0& \text{if $s=\bar s$} \\ <0 & \text{if $s>\bar s$}, \end{cases}
\]
where $\bar s \in \R$ is uniquely defined by
\[
e^{\bar s} = \frac{4 \int_{\R^3} |\nabla w_{a,\mu}|^2}{3\mu\int_{\R^3} w_{a,\mu}^4}.
\]
Recalling that $w_{a,\mu} \in \mathcal{P}(a,\mu)$, see Lemma \ref{lem: natural constraint}, we deduce that $e^{\bar s}=1$, i.e. $\bar s = 0$.
\end{proof}

For $a_1,a_2,\mu_1,\mu_2>0$ let $\beta_1=\beta_1(a_1,a_2,\mu_1,\mu_2)>0$ be defined by the equation:
\begin{equation}\label{def_beta_1}
\max\left\{\frac{1}{a_1^2 \mu_1^2}, \frac{1}{a_2^2 \mu_2^2} \right\} = \frac{1}{a_1^2 (\mu_1+\beta_1)^2} + \frac{1}{a_2^2 (\mu_2+\beta_1)^2}
\end{equation}

\begin{lemma}\label{lem: small 1} For $0<\beta<\beta_1$ there holds:
\[
\inf\left\{J(u,v):(u,v)\in\mathcal{P}(a_1,\mu_1+\beta)\times \mathcal{P}(a_2,\mu_2+\beta)\right\}
 > \max\{ \ell(a_1,\mu_1),\ell(a_2,\mu_2)\}
\]
where $\ell(a_i,\mu_i)$ is defined by \eqref{espressione funionale}.
\end{lemma}

\begin{proof}
Using Young's inequality and recalling the definition of $I_\mu$ (see \eqref{scalar functional}), we obtain for $(u,v) \in \mathcal{P}(a_1,\mu_1+\beta)\times \mathcal{P}(a_2,\mu_2+\beta)$:
\begin{align*}
J(u,v) & = I_{\mu_1}(u) + I_{\mu_2}(v) -\frac{\beta}{2} \int_{\R^3} u^2 v^2 \\
& \ge  I_{\mu_1}(u) + I_{\mu_2}(v)  - \frac{\beta}{4} \int_{\R^3} u^4 - \frac{\beta}{4} \int_{\R^3} v^4 \\
 & =  I_{\mu_1+\beta}(u) + I_{\mu_2+\beta}(v) \ge \inf_{u \in \mathcal{P}(a_1,\mu_1+\beta)}  I_{\mu_1+\beta}(u) + \inf_{v \in \mathcal{P}(a_2,\mu_2+\beta)} I_{\mu_2+\beta}(v) \\
 & = \ell(a_1,\mu_1+\beta) + \ell(a_2,\mu_2+\beta)
\end{align*}
Therefore, the claim is satisfied provided
\[
\max\{\ell(a_1,\mu_1), \ell(a_2,\mu_2)\} < \ell(a_1,\mu_1+\beta) + \ell(a_2,\mu_2+\beta),
\]
that is (by Proposition \ref{prop: explicit})
\begin{equation}\label{eq for beta small 1}
\max\left\{ \frac{C_0 C_1}{8a_1^2 \mu_1^2}, \frac{C_0 C_1}{8a_2^2 \mu_2^2}\right\}
 < \frac{C_0 C_1}{8a_1^2 (\mu_1+\beta)^2} + \frac{C_0 C_1}{8a_2^2 (\mu_2+\beta)^2}.
\end{equation}
Clearly, this holds for $0 < \beta < \beta_1$.
\end{proof}

Now we fix $0<\beta<\beta_1=\beta_1(a_1,a_2,\mu_1,\mu_2)$ and choose $\eps>0$ such that
\begin{equation}\label{eq:small}
\begin{split}
\inf\left\{J(u,v):(u,v)\in\mathcal{P}(a_1,\mu_1+\beta)\times \mathcal{P}(a_2,\mu_2+\beta)\right\} \\
 > \max\{ \ell(a_1,\mu_1),\ell(a_2,\mu_2)\}+\eps.
\end{split}
\end{equation}
We introduce
\begin{equation}\label{notation w_i}
w_1:= w_{a_1,\mu_1 +\beta} \quad \text{and} \quad w_2:= w_{a_2,\mu_2+\beta},
\end{equation}
and for $i=1,2$,
\begin{equation}\label{def fis}
\varphi_i(s):= I_{\mu_i}( s \star w_{i} ) \quad \text{and} \quad \psi_i(s):= \frac{\pa}{\pa s} I_{\mu_i+\beta}(s \star w_i).
\end{equation}

\begin{lemma}\label{lem: small 3}
For $i=1,2$ there exists $\rho_i<0$ and $R_i>0$, depending on $\eps$ and on $\beta$, such that
\begin{itemize}
\item[($i$)] $0< \varphi_i(\rho_i)<\eps$ and $\varphi_i(R_i) \le 0$;
\item[($ii$)] $\psi_i(s)>0$ for any $s <0$ and $\psi_i(s) <0$ for every $s>0$. In particular, $\psi_i(\rho_i) >0$ and $\psi_i(R_i)<0$.
\end{itemize}
\end{lemma}

\begin{proof}
By Lemma \ref{lem: basic properties radial dilation}, we deduce that $\varphi_i(s) \to 0^+$ as $s \to -\infty$, and $\varphi_i(s) \to -\infty$ as $s \to +\infty$. Thus there exist $\rho_i$ and $R_i$ satisfying ($i$). Condition ($ii$) follows directly from Lemma \ref{lem: basic properties radial dilation}.
\end{proof}

Let $Q:= [\rho_1,R_1] \times [\rho_2,R_2]$, and let
\[
\gamma_0(t_1,t_2) := \left( t_1\star w_1, t_2 \star w_2 \right) \in S_{a_1} \times S_{a_2} \qquad \forall (t_1,t_2) \in \overline{Q}.
\]
We introduce the minimax class
\[
\Gamma:= \left\{ \gamma \in \mathcal{C}\left(\overline{Q}, S_{a_1} \times S_{a_2}\right): \text{$\gamma = \gamma_0$ on $\pa Q$} \right\}.
\]
The minimax structure of the problem is enlightened by \eqref{eq:small} and the following two lemmas.

\begin{lemma}\label{lem: minimax inequality}
There holds
\[
\sup_{\pa Q} J(\gamma_0) \le \max\{ \ell(a_1,\mu_1),\ell(a_2,\mu_2)\} + \eps.
\]
\end{lemma}

\begin{proof}
Notice that
\[
J(u,v) = I_{\mu_1}(u) + I_{\mu_2}(v) -\frac{\beta}{2} \int_{\R^3} u^2 v^ 2 \le I_{\mu_1}(u) + I_{\mu_2}(v)
\]
for every $(u,v) \in S_{a_1} \times S_{a_2}$, since $\beta>0$. Therefore, by Lemma \ref{lem: small 3} we infer that
\begin{align*}
J (t_1 \star w_1,\rho_2 \star w_2) & \le I_{\mu_1}(t_1 \star w_1) + I_{\mu_2}(\rho_2 \star w_2) \le I_{\mu_1}(t_1 \star w_1) + \eps \\
& \le \sup_{s \in \R}  I_{\mu_1}(s \star w_1)  + \eps.
\end{align*}
In order to estimate the last term, by Proposition \ref{prop: explicit} it is easy to check that
\[
w_{a_i,\mu_i} = (\bar s_i \star w_i) \quad \text{for} \quad e^{\bar s_i}:= \frac{4 \int_{\R^3}|\nabla w_i|^2}{3 \int_{\R^3} \mu_i w_i^4} = \frac{\mu_i +\beta}{\mu_i}.
\]
As a consequence, observing also that $s_1 \star (s_2 \star w)= (s_1+s_2) \star w$ for every $s_1,s_2 \in \R$ and $w \in H^1(\R^3)$, we have
\begin{equation}\label{remember1}
\sup_{s \in \R} I_{\mu_1}(s \star w_1) = \sup_{s \in \R} I_{\mu_1}(s \star w_{a_1,\mu_1}).
\end{equation}
 As a consequence of Lemma \ref{lem: basic properties radial dilation} the supremum on the right hand side is achieved for $s=0$, and hence
\begin{equation}\label{small: upper est 1}
J (t_1 \star w_1,\rho_2 \star w_2) \le \ell(a_1,\mu_1) + \eps \qquad \forall t_1 \in [\rho_1,R_1],
\end{equation}
and in a similar way one can show that
\begin{equation}\label{small: upper est 2}
J (\rho_1 \star w_1,t_2 \star w_2) \le \ell(a_2,\mu_2) +\eps \qquad \forall t_2 \in [\rho_2,R_2].
\end{equation}
The value of $J(\gamma_0)$ on the remaining sides of $\pa Q$ is smaller: indeed by Lemma \ref{lem: small 3}  and \eqref{remember1}
\begin{equation}\label{small: upper est 3}
\begin{split}
J (t_1 \star w_1,R_2 \star w_2) &\le I_{\mu_1}(t_1 \star w_1) + I_{\mu_2}(R_2 \star w_2)  \\
& \le \sup_{s \in \R} I_{\mu_1}(s \star w_1) = \ell (a_1,\mu_1)
\end{split}
\end{equation}
for every $t_1 \in [\rho_1,R_1]$, and analogously
\begin{equation}\label{small: upper est 4}
J (R_1 \star w_1,t_2 \star w_2) \le \ell (a_2,\mu_2) \qquad \forall t_2 \in [\rho_2,R_2].
\end{equation}
Collecting together \eqref{small: upper est 1}-\eqref{small: upper est 4}, the thesis follows.
\end{proof}


Now we show that the class $\Gamma$ ``links" with $\mathcal{P}(a_1,\mu_1+\beta) \times \mathcal{P}(a_2,\mu_2+\beta)$.
\begin{lemma}\label{lem: linking condition}
For every $\gamma \in \Gamma$, there exists $(t_{1,\gamma},t_{2,\gamma}) \in Q$ such that $\gamma(t_{1,\gamma},t_{2,\gamma}) \in \mathcal{P}(a_1,\mu_1+\beta) \times \mathcal{P}(a_2,\mu_2+\beta)$.
\end{lemma}

\begin{proof}
For $\gamma \in \Gamma$, we use the notation $\gamma(t_1,t_2) = (\gamma_1(t_1,t_2),\gamma_2(t_1,t_2)) \in S_{a_1} \times S_{a_2}$. Let us consider the map $F_{\gamma}:Q \to \R^2$ defined by
\[
F_\gamma(t_1,t_2):= \left( \left. \frac{\pa}{\pa s} I_{\mu_1+\beta}(s \star \gamma_1(t_1,t_2))\right|_{s=0}, \left. \frac{\pa}{\pa s} I_{\mu_2+\beta}(s \star \gamma_2(t_1,t_2))\right|_{s=0} \right).
\]
From
\[
\begin{aligned}
&\left. \frac{\pa}{\pa s} I_{\mu_i+\beta}(s \star \gamma_i(t_1,t_2))\right|_{s=0} \\
&\hspace{1cm}
 = \left. \frac{\pa }{\pa s}\left( \frac{e^{2s}}{2}\int_{\R^3} |\nabla \gamma_i(t_1,t_2)|^2
    - \frac{e^{3s}}{4}(\mu_i+\beta)\int_{\R^3} \gamma_i^4(t_1,t_2)\right) \right|_{s=0} \\
&hspace{1cm}
 = \int_{\R^3}|\nabla \gamma_i(t_1,t_2)|^2 - \frac34(\mu_i+\beta)\int_{\R^3}\gamma_i^4(t_1,t_2)
\end{aligned}
\]
we deduce that
\[
F_\gamma(t_1,t_2) = (0,0) \quad \text{if and only if} \quad \gamma(t_1,t_2) \in \mathcal{P}(a_1,\mu_1+\beta) \times \mathcal{P}(a_2,\mu_2+\beta).
\]
In order to show that $F_\gamma(t_1,t_2) = (0,0)$ has a solution in $Q$ for every $\gamma \in \Gamma$, we can check that the oriented path $F_\gamma(\pa^+Q)$ has winding number equal to $1$ with respect to the origin of $\R^2$, so that standard degree theory gives the desired result. In doing this, we observe at first that $F_\gamma(\pa^+ Q) = F_{\gamma_0}(\pa^+ Q)$ depends only on the choice of $\gamma_0$, and not on $\gamma$. Then we compute
\begin{multline*}
F_{\gamma_0}(t_1,t_2) = \left( e^{2t_1}\left( \int_{\R^3} |\nabla w_1|^2 - \frac{3e^{t_1}}{4}(\mu_1+\beta)\int_{\R^3}  w_1^4\right), \right. \\
\left. e^{2t_2}\left( \int_{\R^3} |\nabla w_2|^2 - \frac{3e^{t_2}}{4}(\mu_1+\beta)\int_{\R^3}  w_2^4\right) \right) = ( \psi_1(t_1), \psi_2(t_2) ),
\end{multline*}
where we recall that the definition of $\psi_i$ has been given in \eqref{def fis}. Therefore, the restriction of $F_{\gamma_0}$ on $\partial Q$ is completely described by Lemma \ref{lem: small 3}-($ii$), see the picture below:
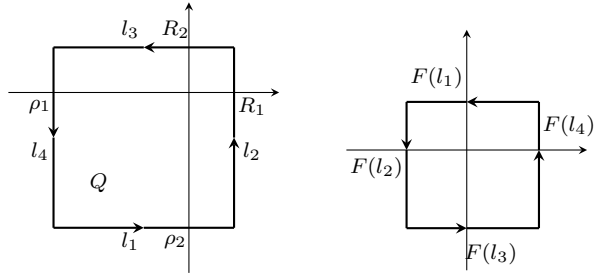
\begin{center}
\begin{tikzpicture}[>= stealth, scale=0.6]
\draw[->] (-3,1)--(3,1);
\draw[->] (1,-3)--(1,3);
\draw[->, thick] (-2,-2)--(0,-2);
\draw[thick] (0,-2) -- (2,-2);
\draw[->, thick] (2,-2)--(2,0);
\draw[thick] (2,0) -- (2,2);
\draw[->, thick] (2,2)--(0,2);
\draw[thick] (0,2) -- (-2,2);
\draw[->, thick] (-2,2)--(-2,0);
\draw[thick] (-2,0) -- (-2,-2);
\draw (-2.3,0.7)[font=\footnotesize] node{$\rho_1$};
\draw (0.7,-2.3)[font=\footnotesize] node{$\rho_2$};
\draw (2.4,0.7)[font=\footnotesize] node{$R_1$};
\draw (0.7,2.4)[font=\footnotesize] node{$R_2$};
\draw (-2.3,-0.3)[font=\footnotesize] node{$l_4$};
\draw (-0.3,-2.3)[font=\footnotesize] node{$l_1$};
\draw (-0.3,2.4)[font=\footnotesize] node{$l_3$};
\draw (2.4,-0.3)[font=\footnotesize] node{$l_2$};
\draw (-1,-1)[font=\footnotesize] node{$Q$};
\end{tikzpicture}
\qquad \begin{tikzpicture}[>= stealth, scale=0.8]
\draw[->] (-2,0)--(2,0);
\draw[->] (0,-2)--(0,2);
\draw[->, thick] (1.2,0.8)--(0,0.8);
\draw[thick] (0,0.8) -- (-1,0.8);
\draw (-0.5,1.2)[font=\footnotesize] node{$F(l_1)$};
\draw[->, thick] (-1,0.8)--(-1,0);
\draw[thick] (-1,0) -- (-1,-1.3);
\draw (-1.5,-0.3)[font=\footnotesize] node{$F(l_2)$};
\draw[->, thick] (-1,-1.3)--(0,-1.3);
\draw[thick] (0,-1.3) -- (1.2,-1.3);
\draw (0.4,-1.73)[font=\footnotesize] node{$F(l_3)$};
\draw[->, thick] (1.2,-1.3)--(1.2,0);
\draw[thick] (1.2,0) -- (1.2,0.8);
\draw (1.7,0.4)[font=\footnotesize] node{$F(l_4)$};
\end{tikzpicture}
\end{center}
In particular, we have that the topological degree
\[
\deg( F_{\gamma},Q,(0,0)) = \iota(F_{\gamma_0}(\pa^+ Q), (0,0) ) = 1,
\]
where $\iota(\sigma,P)$ denotes the winding number of the curve $\sigma$ with respect to the point $P$.
Hence there exists $(t_{1,\gamma},t_{2,\gamma}) \in Q$ such that $F_{\gamma}(t_{1,\gamma},t_{2,\gamma}) = (0,0)$, which, as observed, is the desired result.
\end{proof}

Lemmas \ref{lem: minimax inequality} and \ref{lem: linking condition} permit to apply the minimax principle (Theorem 3.2 in \cite{Ghoussoub}) to $J$ on $\Gamma$. In this way, we could obtain a Palais-Smale sequence for the constrained functional $J$ on $S_{a_1} \times S_{a_2}$, but the boundedness of the Palais-Smale sequence would be unknown. In order to find a bounded Palais-Smale sequence, we shall adapt the trick introduced by one of the authors in \cite{Jeanjean} in the present setting.

\begin{lemma}\label{lem: application minimax}
There exists a Palais-Smale sequence $(u_n,v_n)$ for $J$ on $S_{a_1} \times S_{a_2}$ at the level
\[
c := \inf_{\gamma \in \Gamma} \max_{(t_1,t_2) \in Q} J(\gamma(t_1,t_2)) > \max \{\ell(a_1,\mu_1), \ell(a_2,\mu_2)\},
\]
satisfying the additional condition
\begin{equation}\label{weak Pohozaev}
\int_{\R^3} \left(|\nabla u_n|^2 + |\nabla v_n|^2\right) - \frac{3}{4}\left( \int_{\R^3} \mu_1 u_n^4 + \mu_2 v_n^4 + 2 \beta u_n^2 v_n^2 \right) = o(1),
\end{equation}
where $o(1) \to 0$ as $n \to \infty$. Furthermore, $u_n^-,v_n^- \to 0$ a.e. in $\R^3$ as $n \to \infty$.
\end{lemma}

\begin{proof}
We consider the augmented functional $\tilde J: \R \times S_{a_1} \times S_{a_2} \to \R$ defined by $\tilde J(s, u,v) := J(s \star u, s\star v)$. Let also
\[
\tilde \gamma_0(t_1,t_2):= (0,\gamma_0(t_1,t_2)) = (0,t_1 \star w_1, t_2 \star w_2),
\]
and
\[
\tilde \Gamma:= \left\{ \tilde \gamma \in \mathcal{C}(Q, \R \times S_{a_1} \times S_{a_2}): \tilde \gamma = \tilde \gamma_0 \text{ on $\pa Q$} \right\}.
\]
We wish to apply the minimax principle Theorem 3.2 in \cite{Ghoussoub} to the functional $\tilde J$ with the minimax class $\tilde \Gamma$, in order to find a Palais-Smale sequence for $\tilde J$ at level
\[
\tilde c:= \inf_{ \tilde \gamma \in \tilde \Gamma} \sup_{(t_1,t_2) \in Q} \tilde J(\tilde \gamma(t_1,t_2)).
\]
Notice that, since $\tilde J(\tilde \gamma_0) = J(\gamma_0)$ on $\partial Q$, by Lemmas \ref{lem: minimax inequality} and \ref{lem: linking condition}, the assumptions of the minimax principle will be satisfied if we show that $\tilde c = c$. This equality is a simple consequence of the definition: firstly, since $\Gamma \subset \tilde \Gamma$, we have $\tilde c \le c$. Secondly, using the notation
\[
\tilde \gamma(t_1,t_2) = (s(t_1,t_2), \gamma_1(t_1,t_2), \gamma_2(t_1,t_2) ),
\]
for any $\tilde \gamma \in \tilde \Gamma$ and $(t_1,t_2) \in Q$ it results that
\[
\tilde J(\tilde \gamma(t_1,t_2)) = J(s(t_1,t_2) \star \gamma_1(t_1,t_2), s(t_1,t_2) \star \gamma_2(t_1,t_2)),
\]
and $(s(\cdot) \star \gamma_1(\cdot), s(\cdot) \star \gamma_2(\cdot)) \in \Gamma$. Thus $\tilde c = c$, and the minimax principle is applicable.

Notice that, using the notation of Theorem 3.2 in \cite{Ghoussoub}, we can choose the minimizing sequence $\tilde \gamma_n=(s_n, \gamma_{1,n},\gamma_{2,n})$ for $\tilde c$ satisfying the additional conditions $s_n \equiv 0$, $\gamma_{1,n}(t_1,t_2) \ge 0$ a.e. in $\R^N$ for every $(t_1,t_2) \in Q$, $\gamma_{2,n}(t_1,t_2) \ge 0$ a.e. in $\R^N$ for every $(t_1,t_2) \in Q$. Indeed, the first condition comes from the fact that
\begin{align*}
\tilde J(\tilde \gamma(t_1,t_2)) & = J(s(t_1,t_2) \star \gamma_1(t_1,t_2), s(t_1,t_2) \star \gamma_2(t_1,t_2)) \\
& =\tilde  J(0, s(t_1,t_2) \star \gamma_1(t_1,t_2), s(t_1,t_2) \star \gamma_2(t_1,t_2)).
\end{align*}
The remaining ones are a consequence of the fact that $\tilde J(s,u,v) = \tilde J(s,|u|, |v|)$.

In conclusion, Theorem 3.2 in \cite{Ghoussoub} implies that there exists a Palais-Smale sequence $(\tilde s_n,\tilde u_n, \tilde v_n)$ for $\tilde J$ on $\R \times S_{a_1} \times S_{a_2}$ at level $\tilde c$, and such that
\begin{equation}\label{localization}
\lim_{n \to \infty} |\tilde s_n| + \dist_{H^1}\left( (\tilde u_n,\tilde v_n), \tilde \gamma_n(Q) \right)
= 0.
\end{equation}
To obtain a Palais-Smale sequence for $J$ at level $c$ satisfying \eqref{weak Pohozaev}, it is possible to argue as in \cite[Lemma 2.4]{Jeanjean} with minor changes. The fact that $u_n^-,v_n^- \to 0$ a.e. in $\R^N$ as $n \to \infty$ comes from \eqref{localization}. Finally, the lower estimate for $c$ comes from Lemma \ref{lem: minimax inequality}.
\end{proof}

To complete the proof of Theorem \ref{thm: existence for beta small}, we aim at showing that $(u_n,v_n)$ is strongly convergent in $H^1(\R^3,\R^2)$ to a limit $(u,v)$. Once this has been achieved the claim follows because
\[
dJ |_{S_{a_1} \times S_{a_2}}(u_n,v_n) \to 0 \quad \text{and} \quad (u_n,v_n) \in S_{a_1} \times S_{a_2}
\]
for all $n$. A first step in this direction is given by the following statement.

\begin{lemma}\label{lem: bounds on PS}
The sequence $\{(u_n,v_n)\}$ is bounded in $H^1(\R^3,\R^2)$. Furthermore, there exists $\bar C>0$ such that
\[
\int_{\R^3} |\nabla u_n|^2 + |\nabla v_n|^2 \ge \bar C \qquad \text{for all }n.
\]
\end{lemma}
\begin{proof}
Using \eqref{weak Pohozaev}, we have
\[
J(u_n,v_n) = \frac{1}{6} \left(\int_{\R^3} |\nabla u_n|^2 + |\nabla v_n|^2\right)  -o(1),
\]
where $o(1) \to 0$ as $n \to \infty$. Therefore, the desired results follow from the fact that $J(u_n,v_n) \to c > 0$.
\end{proof}

By the previous lemma, up to a subsequence $(u_n,v_n) \to (\tilde u,\tilde v)$ weakly in $H^1(\R^3)$, strongly in $L^4(\R^3)$ (by compactness of the embedding $H^1_{\rad}(\R^3) \hookrightarrow L^4(\R^3)$), and a.~e.~ in $\R^3$; in particular, both $\tilde u$ and $\tilde v$ are nonnegative in $\R^3$; we explicitly remark that we cannot deduce strong convergence in $L^2(\R^3)$, so that we cannot conclude that $(\tilde u,\tilde v) \in S_{a_1} \times S_{a_2}$. Observe that as a consequence of $dJ |_{S_{a_1} \times S_{a_2}}(u_n,v_n) \to 0$
there exist two sequences of real numbers $(\lambda_{1,n})$ and $(\lambda_{2,n})$ such that
\begin{multline}\label{gradient to zero}
\int_{\R^3} \left( \nabla u_n \cdot  \nabla \varphi + \nabla v_n \cdot \nabla \psi - \mu_1 u_n^3 \varphi -\mu_2 v_n^3 \psi - \beta u_n v_n(u_n \psi + v_n \varphi) \right) \\
-\int_{\R^3} \left(\lambda_{1,n} u_n \varphi +\lambda_{2,n} \psi\right) = o(1) \|(\varphi,\psi)\|_{H^1}
\end{multline}
for every $(\varphi, \psi) \in H^1(\R^3,\R^2)$, with $o(1) \to 0$ as $n \to \infty$. For more details we refer to Lemma 2.2 of \cite{BarJea}.

\begin{lemma}\label{lem: property multipliers}
Both $(\lambda_{1,n})$ and $(\lambda_{2,n})$ are bounded sequences, and at least one of them is converging, up to a subsequence, to a strictly negative value.
\end{lemma}
\begin{proof}
The value of the $(\lambda_{i,n})$ can be found using $(u_n,0)$ and $(0,v_n)$ as test functions in \eqref{gradient to zero}:
\[
\begin{split}
\lambda_{1,n} a_1^2 &= \int_{\R^3} \left(|\nabla u_n|^2 -\mu_1 u_n^4 -\beta u_n^2 v_n^2\right) -o(1) \\
\lambda_{2,n} a_2^2 &= \int_{\R^3} \left(|\nabla v_n|^2 -\mu_2 v_n^4 -\beta u_n^2 v_n^2\right) -o(1),
\end{split}
\]
with $o(1) \to 0$ as $n \to \infty$. Hence the boundedness of $(\lambda_{i,n})$ follows by the boundedness of $(u_n,v_n)$ in $H^1$ and in $L^4$. Moreover, by \eqref{weak Pohozaev} and Lemma \ref{lem: bounds on PS}
\begin{align*}
\lambda_{1,n} a_1^2 + \lambda_{2,n} a_2^2  & = \int_{\R^3} \left( |\nabla u_n|^2 + |\nabla v_n|^2  -\mu_1 u_n^4 -\mu_2 v_n^4 - 2\beta u_n^2 v_n^2\right) -o(1) \\
& = -\frac{1}{3}  \int_{\R^3} \left(|\nabla u_n|^2 + |\nabla v_n|^2\right) +o(1) \le -\frac{\bar C}{6}
\end{align*}
for every $n$ sufficiently large, so that at least one sequence of $(\lambda_{i,n})$ is negative and bounded away from $0$.
\end{proof}

From now on, we consider converging subsequences $\lambda_{1,n} \to \lambda_1 \in \R$ and $\lambda_{2,n} \to \lambda_2 \in \R$. The sign of the limit values plays an essential role in our argument, as clarified by the next statement.

\begin{lemma}\label{lem: strong convergence}
If $\lambda_1<0$ (resp. $\lambda_2<0$) then $u_n \to \bar u$ (resp. $v_n \to \bar v$) strongly in $H^1(\R^3)$. \end{lemma}
\begin{proof}
Let us suppose that $\lambda_1<0$. By weak convergence in $H^1(\R^3)$, strong convergence in $L^4(\R^3)$, and using \eqref{gradient to zero}, we have
\begin{align*}
o(1) &= \left( dJ(u_n,v_n) - dJ(\tilde u,\tilde v) \right)[(u_n-\tilde u,0)] - \lambda_1 \int_{\R^3} (u_n-\tilde u)^2 \\
& = \int_{\R^3} |\nabla (u_n-\tilde u)|^2 - \lambda_1 (u_n-\tilde u)^2 + o(1),
\end{align*}
with $o(1) \to 0$ as $n \to \infty$. Since $\lambda_1<0$, this is equivalent to the strong convergence in $H^1$. The proof in the case $\lambda_2 <0$ is similar.
%
%
\end{proof}

\begin{remark}\label{rem: dependence on beta}
It is important to observe that Lemmas \ref{lem: bounds on PS}-\ref{lem: strong convergence} do not depend on the value of $\beta$. This implies that we can use them in the proof of Theorem \ref{thm: existence for beta large}.
\end{remark}

\begin{proof}[Conclusion of the proof of Theorem \ref{thm: existence for beta small}]
By \eqref{gradient to zero}, the convergence of $(\lambda_{1,n})$ and $(\lambda_{2,n})$, and the weak convergence $(u_n,v_n) \wc (\tilde u,\tilde v)$, we have that $(\tilde u,\tilde v)$ is a solution of \eqref{system}. It remains to prove that it satisfies \eqref{normalization}. Without loss of generality, by Lemma \ref{lem: property multipliers} we can suppose that $\lambda_1<0$, and hence (see Lemma \ref{lem: strong convergence}) $u_n \to \tilde u$ strongly in $H^1(\R^3)$. If $\lambda_2<0$, we can infer in the same way that $v_n \to \tilde v$ strongly in $H^1(\R^3)$, which completes the proof. Now we argue by contradiction and assume that $\lambda_2 \ge 0$ and $v_n \not \to \tilde v$ strongly in $H^1(\R^3)$. Notice that, by regularity, any weak solution of \eqref{system} is smooth. Since both $\tilde u, \tilde v \ge 0$ in $\R^N$, we have that
\[
-\Delta \tilde v = \lambda_2 \tilde v + \mu_2 \tilde v^3 + \beta \tilde u^2 \tilde v \ge 0 \qquad \text{in $\R^3$},
\]
and hence we can apply Lemma A.2 in \cite{Ikoma}, deducing that $\tilde v \equiv 0$. In particular, this implies that $\tilde u$ solves
\begin{equation}\label{eq: scalar nom}
\begin{cases}
-\Delta \tilde u -\lambda_1 \tilde u = \mu_1 \tilde u^3 & \text{in $\R^3$} \\
\tilde u>0 & \text{in $\R^3$} \\
\int_{\R^3} \tilde u^2 = a_1,
\end{cases}
\end{equation}
so that $\tilde u \in \mathcal{P}(a_1,\mu_1)$ and $I_{\mu_1}(\tilde u) = \ell(a_1,\mu_1)$ (recall \eqref{def of P} and Proposition \ref{prop: explicit}). But then, using \eqref{weak Pohozaev} and $\tilde u \in \mathcal{P}(a_1,\mu_1)$, we obtain
\begin{equation}\label{eq37}
\begin{split}
c &= \lim_{n \to \infty} J(u_n,v_n)
   = \lim_{n\to\infty}\frac18\int_{\R^3}\left(\mu_1u_n^4+2\beta u_n^2v_n^2+\mu_2v_n^4\right)\\
 & = \frac{\mu_1}{8} \int_{\R^3} \tilde u^4 = I_{\mu_1}(\tilde u) = \ell(a_1,\mu_1),
\end{split}
\end{equation}
in contradiction with Lemma \ref{lem: application minimax}.
\end{proof}

\begin{remark}\label{rem: on uniqueness system}
In the conclusion of the proof of Theorem \ref{thm: existence for beta small} we used the uniqueness, up to translation, of the positive solution to \eqref{eq: scalar nom} to deduce that, being $\tilde u$ a positive solution of \eqref{eq: scalar nom}, its level $I_{\mu_1}(\tilde u_1)$ is equal to $\ell(a_1,\mu_1)$. Such a uniqueness result is known for systems as \eqref{system} only if $\beta$ is very small (see \cite{IkoUni}). This is what prevents us to extend Theorem \ref{thm: existence for beta small} to systems with several components without requiring the coupling parameters to be very small. In particular, we observe that the minimax construction can be extended to systems with an arbitrary number of components with some extra work.
\end{remark}

\section{Proof of Theorem \ref{thm: existence for beta large}}

This section is divided into two parts. In the first one, we show the existence of a positive solution $(\bar u,\bar v)$, in the second one we characterize it as a ground state, in the sense that
\[
\begin{aligned}
J(\bar u,\bar v)
 &= \inf\{J(u,v): (u,v)\in V\}\\
 &= \inf\left\{J(u,v):\text{$(u,v)$ is a solution of \eqref{system}-\eqref{normalization} for
     some $\lambda_1,\lambda_2$} \right\}.
\end{aligned}
\]

The proof of Theorem \ref{thm: existence for beta large} is based upon a mountain pass argument, and, compared with the proof of Theorem \ref{thm: existence for beta small}, it is closer to the proof of the existence of normalized solutions for the single equation. We shall often consider, for $(u,v) \in S_{a_1} \times S_{a_2}$, the function
\[
 J( s \star (u,v)) = \frac{e^{2s}}{2}\int_{\R^3} \left(|\nabla u|^2 + |\nabla v|^2\right) - \frac{e^{3s}}{4}\int_{\R^3} \left(\mu_1 u^4 + 2\beta u^2 v^2 + \mu_2 v^4 \right),
\]
where $s \star (u,v) = (s \star u, s \star v)$ for short, and $s \star u$ is defined in \eqref{def star}. Recall that if $(u,v) \in S_{a_1} \times S_{a_2}$, then also $s \star (u,v) \in S_{a_1} \times S_{a_2}$.
As an immediate consequence of the definition, the following holds:

\begin{lemma}\label{lem: easy computation}
Let $(u,v) \in S_{a_1} \times S_{a_2}$. Then
\[
\lim_{s \to -\infty} \int_{\R^3} |\nabla (s \star u)|^2 + |\nabla (s \star v)|^2 = 0, \quad
\lim_{s \to +\infty} \int_{\R^3} |\nabla (s \star u)|^2 + |\nabla (s \star v)|^2 = +\infty,
\]
and
\[
\lim_{s \to -\infty}  J( s \star (u,v)) = 0^+, \quad
\lim_{s \to -\infty}  J( s \star (u,v)) = -\infty.
\]
\end{lemma}

The next lemma enlighten the mountain pass structure of the problem.

\begin{lemma}\label{lem: minimax inequality 2}
There exists $K>0$ sufficiently small such that for the sets
\[
A:= \left\{ (u,v) \in S_{a_1} \times S_{a_2} : \int_{\R^3} |\nabla u|^2+|\nabla v|^2 \le K \right\}  \]
and
\[
B:= \left\{ (u,v) \in S_{a_1} \times S_{a_2} : \int_{\R^3} |\nabla u|^2+|\nabla v|^2 = 2K \right\} \]
there holds
\[
J(u) > 0 \text{ on } A  \quad \text{and} \quad \sup_A J < \inf_B J.
\]
\end{lemma}
\begin{proof}
By the Gagliardo-Nirenberg inequality \eqref{Gagliardo-Nirenberg}
\[
\int_{\R^3}  \left(\mu_1 u^4 + 2\beta u^2 v^2 + \mu_2 v^4 \right)
 \le C \int_{\R^3} \left( u^4 + v^4\right)
 \le C \left(\int_{\R^3} |\nabla u|^2 + |\nabla v|^2 \right)^{3/2}
\]
for every $(u,v) \in S_{a_1} \times S_{a_2}$, where $C>0$ depends on $\mu_1,\mu_2,\beta,a_1,a_2>0$ but not on the particular choice of $(u,v)$. Now, if $(u_1,v_1) \in B$ and $(u_2,v_2) \in A$ (with $K$ to be determined), we have
\begin{align*}
J(u_1,v_1) - J(u_2,v_2)
& \ge \frac12\left( \int_{\R^3} |\nabla u_1|^2 + |\nabla v_1|^2 - \int_{\R^3} |\nabla u_2|^2
       + |\nabla v_2|^2\right) \\
&\hspace{1cm}
 -\frac14 \int_{\R^3}\left(\mu_1 u_1^4 + 2\beta u_1^2 v_1^2 + \mu_2 v_1^4 \right) \\
& \ge \frac{K}{2} - \frac{C}{4} (2K)^{3/2} \ge \frac{K}{4}
\end{align*}
provided $K>0$ is sufficiently small. Furthermore, making $K$ smaller if necessary, we have also
\begin{equation}\label{in A >0}
J(u_2,v_2) \ge \frac{1}{2}\left( \int_{\R^3} |\nabla u_2|^2 + |\nabla v_2|^2 \right) - \frac{C}{4} \left( \int_{\R^3} |\nabla u_2|^2 + |\nabla v_2|^2 \right)^{3/2} >0
\end{equation}
for every $(u_2,v_2) \in A$.
\end{proof}

In order to introduce a suitable minimax class, we recall that $w_{a_1,\mu_1}$ (resp. $w_{a_2,\mu_2}$) is the unique positive solution of \eqref{scalar pb with mass} with $a=a_1$ and $\mu=\mu_1$ (resp. $a=a_2$ and $\mu=\mu_2$). Now we define
\begin{equation}\label{def C}
C:= \left\{ (u,v) \in S_{a_1} \times S_{a_2}: \int_{\R^3} |\nabla u|^2 +|\nabla v|^2 \ge 3K \text{ and } J(u,v) \le 0\right\}.
\end{equation}
It is clear by Lemma \ref{lem: easy computation} that there exist $s_1<0$ and $s_2>0$ such that
\[
s_1 \star (w_{a_1,\mu_1}, w_{a_2,\mu_2})=: (\bar u_1,\bar v_1) \in A\quad \text{and} \quad s_2 \star (w_{a_1,\mu_1}, w_{a_2,\mu_2})=: (\bar u_2,\bar v_2) \in C.
\]
Finally we define
\begin{equation}\label{def: Gamma 2}
\Gamma:= \left\{ \gamma \in \mathcal{C}([0,1],S_{a_1} \times S_{a_2}): \gamma(0) = (\bar u_1,\bar v_1) \text{ and } \gamma(1) = (\bar u_2,\bar v_2) \right\}.
\end{equation}
By Lemma \ref{lem: minimax inequality 2} and by the continuity of the $L^2$-norm of the gradient in the topology of $H^1$, it follows that the mountain pass lemma is applicable for $J$ on the minimax class $\Gamma$. Arguing as in Lemma \ref{lem: application minimax}, we deduce the following:

\begin{lemma}\label{lem: appl minimax 2}
There exists a Palais-Smale sequence $(u_n,v_n)$ for $J$ on $S_{a_1} \times S_{a_2}$ at the level
\[
d:= \inf_{\gamma \in \Gamma} \max_{t \in [0,1]} J(\gamma(t)),
\]
satisfying the additional condition \eqref{weak Pohozaev}:
\[
\int_{\R^3} \left(|\nabla u_n|^2 + |\nabla v_n|^2\right) - \frac{3}{4}\left( \int_{\R^3} \mu_1 u_n^4 + \mu_2 v_n^4 + 2 \beta u_n^2 v_n^2 \right) = o(1),
\]
with $o(1) \to 0$ as $n \to \infty$. Furthermore, $u_n^-,v_n^- \to 0$ a.e. in $\R^3$ as $n \to \infty$.
\end{lemma}

As in the previous section, the last part of the proof consists in showing that $(u_n,v_n) \to (\bar u,\bar v)$ in $H^1(\R^3,\R^2)$, and $(\bar u,\bar v)$ is a solution of \eqref{system}-\eqref{normalization}. This can be done similarly to the case $\beta>0$ small, recalling also Remark \ref{rem: dependence on beta}. Firstly, thanks to \eqref{weak Pohozaev}, up to a subsequence $(u_n,v_n) \to (\bar u,\bar v)$ weakly in $H^1(\R^3,\R^2)$, strongly in $L^4(\R^3,\R^2)$, a.~e.\ in $\R^3$. By weak convergence and by Lemma \ref{lem: property multipliers}, $(u,v)$ is a solution of \eqref{system} for some $\lambda_1,\lambda_2 \in \R$. Moreover, we can also suppose that one of these parameters, say $\lambda_1$, is strictly negative. Thus, Lemma \ref{lem: strong convergence} implies that $u_n \to \bar u$ strongly in $H^1(\R^3)$. If by contradiction $v_n \not \to \bar v$ strongly in $H^1(\R^3)$, then $\lambda_2 \ge 0$, and by Lemma A.2 in \cite{Ikoma} we deduce that $\bar v \equiv 0$. As in \eqref{eq37}, this implies that $d =\ell(a_1,\mu_1)$ (defined in Proposition \ref{prop: explicit}), and it remains to show that this gives a contradiction, which is the object of the following lemma.

\begin{lemma}\label{lem: crucial estimate beta_2}
There exists $\beta_2>0$ sufficiently large such that
\[
\sup_{s \in \R} J(s \star (w_{a_1,\mu_1},w_{a_2,\mu_2})) < \min\{ \ell(a_1,\mu_1), \ell(a_2,\mu_2)\} \qquad \text{for all }\beta > \beta_2.
\]
\end{lemma}
\begin{proof}
By Lemma \ref{lem: basic properties radial dilation}, for any $\eps>0$ there exists $s_\eps \ll -1$ such that
\[
I_{\mu_1}(s \star w_{a_1,\mu_1}) + I_{\mu_2}(s \star w_{a_2,\mu_2}) < \varepsilon
 \qquad \text{for all }s <s_\eps.
\]
For such values of $s$ we have $J(s \star (w_{a_1,\mu_1},w_{a_2,\mu_2})) <\eps$, because $\beta>0$. If $s \ge s_\eps$, then the interaction term can be bounded from below:
\[
\int_{\R^3} (s \star w_{a_1,\mu_1})^2 (s \star w_{a_2,\mu_2})^2 = e^{3s} \underbrace{\int_{\R^3} w_{a_1,\mu_1}^2 w_{a_2,\mu_2}^2}_{=: C_2 = C_2(a_1,a_2,\mu_1,\mu_2)>0} \ge C_2 e^{3s_{\eps}}.
\]
As a consequence, recalling that $\sup_s  I_{\mu_i}(s \star w_{a_i,\mu_i}) =  I_{\mu_i}(w_{a_i,\mu_i}) = \ell(a_i,\mu_i)$ (see again Lemma \ref{lem: minimax inequality}), we have
\begin{align*}
 J(s \star (w_{a_1,\mu_1},w_{a_2,\mu_2})) & \le I_{\mu_1}(s \star w_{a_1,\mu_1}) + I_{\mu_2}(s \star w_{a_2,\mu_2}) - \frac{C_2}{2} e^{3 s_{\eps}} \beta   \\\
 & \le  \ell(a_1,\mu_1) + \ell(a_2,\mu_2) - \frac{C_2}{2} e^{3 s_{\eps}} \beta,
 \end{align*}
 and the last term is strictly smaller than $\min\{ \ell(a_1,\mu_1), \ell(a_2,\mu_2)\}$ provided $\beta$ is sufficiently large.
\end{proof}

\begin{remark}\label{rem: estimate beta_2}
From the previous proof one can obtain an explicit estimate of $\beta_2$ in Theorem \ref{thm: existence for beta large}, in the following way. Choose as $\eps$ any value smaller than $\min\{ \ell(a_1,\mu_1), \ell(a_2,\mu_2)\}$. Then one can explicitly estimate $s_\eps$ using Lemma \ref{lem: basic properties radial dilation} (the smaller is $\eps$, the larger is $|s_\eps|$). Once that $\eps$ is fixed and $s_\eps$ is estimated, it remains to solve the inequality
\[
 \ell(a_1,\mu_1) + \ell(a_2,\mu_2) - \frac{C_2(a_1,a_2,\mu_1,\mu_2)}{2} e^{3 s_{\eps}} \beta < \min\{ \ell(a_1,\mu_1), \ell(a_2,\mu_2)\}
\]
with respect to $\beta$. An optimization in $0<\eps< \min\{ \ell(a_1,\mu_1), \ell(a_2,\mu_2)\}$ reveals that $\beta_2$ can be chosen as
\begin{multline*}
\beta_2 = \left[ \ell(a_1,\mu_1) + \ell(a_2,\mu_2) - \min\{ \ell(a_1,\mu_1), \ell(a_2,\mu_2)\} \right] \\
\cdot \left.\frac{2e^{-3s_\eps} }{C_2(a_1,a_2,\mu_1,\mu_2)} \right|_{\eps =  \min\{ \ell(a_1,\mu_1), \ell(a_2,\mu_2)\} }.
\end{multline*}
\end{remark}

\begin{proof}[Existence of a positive solution at level $d$] In our contradiction argument, we are supposing that $v_{n} \not \to \bar v$ strongly in $H^1(\R^3)$, and hence we have observed that $\bar v \equiv 0$ and $d= \ell(a_1,\mu_1)$. Let us consider the path
\[
\gamma(t) := \left( ((1-t) s_1 + t s_2) \star (w_{a_1,\mu_1},w_{a_2,\mu_2}) \right).
\]
Clearly, $\gamma \in \Gamma$, so that by Lemma \ref{lem: crucial estimate beta_2}
\[
d \le \sup_{t \in [0,1]} J(\gamma(t)) \le \sup_{s \in \R} J(s \star (w_{a_1,\mu_1},w_{a_2,\mu_2})) < \ell (a_1,\mu_1),
\]
a contradiction.
\end{proof}

Let us now turn to the variational characterization for $(\bar u,\bar v)$.
In what follows we aim at proving that
\[
\begin{aligned}
J(\bar u,\bar v)
 &= \inf\{J(u,v): (u,v)\in V\}\\
 &= \inf\left\{J(u,v):\text{$(u,v)$ is a solution of \eqref{system}-\eqref{normalization} for
     some $\lambda_1,\lambda_2$} \right\}
\end{aligned}
\]
Let us recall the definitions of $A$, see Lemma \ref{lem: minimax inequality 2}, and of $C$, see \eqref{def C}. We set
\[
A^+:= \{(u,v) \in A: u,v \ge 0 \text{ a.e. in $\R^3$}\}
\]
and
\[
C^+:= \{(u,v) \in C: u,v \ge 0 \text{ a.e. in $\R^3$}\}.
\]
For any $(u_1,v_1) \in A^+$ and $(u_2,v_2) \in C^+$, let
\[
\Gamma(u_1,v_1,u_2,v_2):= \left\{ \gamma \in \mathcal{C}([0,1],S_{a_1} \times S_{a_2}): \gamma(0) = (u_1,v_1) \text{ and } \gamma(1) = (u_2,v_2) \right\}.
\]

\begin{lemma}\label{lem: connectedness}
The sets $A^+$ and $C^+$ are connected by arcs, so that
\begin{equation}\label{c = something}
d = \inf_{\gamma \in \Gamma(u_1,v_1,u_2,v_2)} \max_{t \in [0,1]} J(\gamma(t))
\end{equation}
for every $(u_1,v_1) \in A^+$ and $(u_2,v_2) \in C^+$.
\end{lemma}
\begin{proof}
The proof is similar to the one of Lemma 2.8 in \cite{Jeanjean}. Equality \eqref{c = something} follows easily, once we show that $A^+$ and $C^+$ are connected by arcs (recall the definition of $\Gamma$, see \eqref{def: Gamma 2}, and also that $\bar u_1,\bar v_1,\bar u_2,\bar v_2 \ge 0$ in $\R^N$).

Let $(u_1,v_1), (u_2,v_2) \in S_{a_1} \times S_{a_2}$ be nonnegative functions such that
\begin{equation}\label{additional assumptions}
\int_{\R^3} |\nabla u_1|+ |\nabla v_1|^2 = \int_{\R^3} |\nabla u_2|+ |\nabla v_2|^2 = \alpha^2
\end{equation}
for some $\alpha>0$. We define, for $s \in \R$ and $t \in [0,\pi/2]$,
\[
h(s,t)(x):= \left( \cos t(s \star u_1)(x) + \sin t (s \star u_2)(x), \cos t(s \star v_1)(x) + \sin t (s \star v_2)(x) \right).
\]
Although $h$ depends on $(u_1,v_1)$ and $(u_2,v_2)$, we will not stress such dependence in order to simplify the notation. Setting $h=(h_1,h_2)$, we have that $h_1(s,t),h_2(s,t) \ge 0$ a.e. in $\R^N$, and by direct computations it is not difficult to check that
\begin{align*}
\int_{\R^3} h_1^2(s,t) & = a_1^2+ \sin(2t) \int_{\R^3} u_1 u_2 \\
\int_{\R^3} h_2^2(s,t) & = a_2^2+ \sin(2t) \int_{\R^3} v_1 v_2 \\
\int_{\R^3} |\nabla h_1(s,t)|+ |\nabla h_2(s,t)|^2 &= e^{2s}\left( \alpha^2 +\sin(2t) \int_{\R^3} \nabla u_1 \cdot \nabla u_2 + \nabla v_1 \cdot \nabla v_2 \right)
\end{align*}
for all $(s,t) \in \R \times [0,\pi/2]$. From these expressions, and recalling \eqref{additional assumptions} and the fact that $u_1,v_1,u_2,v_2 \ge 0$ a.~e.\ in $\R^N$, it is possible to deduce that there exists $C>0$ (depending on $(u_1,v_1)$ and $(u_2,v_2)$) such that
\begin{align*}
 C e^{2s} \le \int_{\R^3} |\nabla h_1(s,t)|^2+ |\nabla h_2(s,t)|^2 \le 2 \alpha^2 e^{2s} \\
a_1^2 \le \int_{\R^3} h_1^2(s,t) \le 2a_1^2 \quad \text{and}  \quad a_2^2 \le \int_{\R^3} h_2^2(s,t) \le 2a_2^2
\end{align*}
Thus, we can define for $(s,t) \in \R \times [0,\pi/2]$ the function
\[
\hat h(s,t)(x) := \left( a_1 \frac{h_1(s,t)}{\|h_1(s,t)\|_{L^2}}, a_2 \frac{h_2(s,t)}{\|h_2(s,t)\|_{L^2}} \right).
\]
Notice that $\hat h(s,t) \in S_{a_1} \times S_{a_2}$ for every $(s,t)$. It results
\begin{equation}\label{estimate gradient h bar}
\frac{\min\{a_1^2, a_2^2\} C e^{2s} }{2 \max \{ a_1^2,a_2^2\} } \le \int_{\R^3} |\nabla \hat h_1(s,t)|^2+ |\nabla \hat h_2(s,t)|^2 \le \frac{2\alpha^2 e^{2s} \max \{ a_1^2, a_2^2\}}{\min\{a_1^2,a_2^2\}}.
\end{equation}
Furthermore, using again \eqref{additional assumptions} (and replacing if necessary $C$ with a smaller quantity), it is possible to check that
\begin{equation}\label{estimate 4 h bar}
\int_{\R^3} \hat h_1^4(s,t)  \ge C e^{3s} \quad \text{and} \quad \int_{\R^3} \hat h_2^4(s,t)  \ge C e^{3s}
\end{equation}
for all $(s,t) \in \R \times [0,\pi/2]$.

These estimates permits to prove the desired result. Let $(u_1,v_1),(u_2,v_2) \in A^+$, and let $\hat h$ de defined as in the previous discussion. By \eqref{estimate gradient h bar} there exists $s_0 >0$ such that
\[
\int_{\R^3} |\nabla \hat h_1(-s_0,t)|^2+ |\nabla \hat h_2(-s_0,t)|^2 \le K
 \qquad \text{for all }t \in \left[0,\frac{\pi}{2}\right],
\]
where $K$ has been defined in Lemma \ref{lem: minimax inequality 2}. For this choice of $s_0$, let
\[
\sigma_1(r):= \begin{cases} -r \star (u_1,v_1) = \hat h(-r,0) & 0 \le r \le s_0   \\
\hat h(-s_0,r-s_0) & s_0 < r \le s_0 + \frac{\pi}{2} \\
\left(r-2s_0- \frac{\pi}{2}\right) \star (u_2,v_2) = \hat h\left(r-2s_0-\frac{\pi}{2},\frac{\pi}{2}\right) & s_0 +\frac{\pi}{2} < r \le 2s_0+\frac{\pi}{2}.
\end{cases}
\]
It is not difficult to check that $\sigma$ is a continuous path connecting $(u_1,v_1)$ and $(u_2,v_2)$ and lying in $A^+$. To conclude that $A^+$ is connected by arcs, it remains to analyse the cases when condition \eqref{additional assumptions} is not satisfied. Suppose for instance
\[
\int_{\R^3} |\nabla u_1|^2+ |\nabla v_1|^2 > \int_{\R^3} |\nabla u_2|^2+ |\nabla v_2|^2.
\]
Then, by Lemma \ref{lem: easy computation}, there exists $s_1<0$ such that
\[
\int_{\R^3} |\nabla (s_1 \star u_1)|^2+ |\nabla (s_1 \star v_1)|^2 = \int_{\R^3} |\nabla u_2|^2+ |\nabla v_2|^2.
\]
Therefore, to connect $(u_1,v_1)$ and $(u_2,v_2)$ by a path in $A^+$ we can at first connect $(u_1,v_1)$ with $s_1 \star (u_1,v_1)$, and then connect this point with $(u_2,v_2)$.

To prove that also $C^+$ is connected by arcs, let us fix $(u_1,v_1),(u_2,v_2) \in C^+$, and suppose that \eqref{additional assumptions} holds (as before, we can always reduce to this case). By \eqref{estimate gradient h bar} and \eqref{estimate 4 h bar}, there exists $s_0>0$ such that
\[
\int_{\R^3} |\nabla \hat  h_1(s_0,t)|^2 + |\nabla \hat h_2(s_0,t)|^2 \ge 3K \quad \text{and} \quad J(\hat h(s_0,t)) \le 0
\]
for all $t \in [0,\pi/2]$. For this choice of $s_0$, we set
\[
\sigma_2(r):= \begin{cases} r \star (u_1,v_1) = \hat h(r,0) & 0 \le r \le s_0   \\
\hat h(s_0,r-s_0) & s_0 < r \le s_0 + \frac{\pi}{2} \\
\left(2s_0+\frac{\pi}{2}-r\right) \star (u_2,v_2) = \hat h\left(2s_0+\frac{\pi}{2}-r,\frac{\pi}{2}\right) & s_0 +\frac{\pi}{2} < r \le 2s_0+\frac{\pi}{2},
\end{cases}
\]
which is the desired continuous path connecting $(u_1,v_1)$ and $(u_2,v_2)$ in $C^+$.
\end{proof}

As we shall see, the previous lemma will be the key in proving the variational characterization of $(\bar u,\bar v)$. Let us recall the set
\[
V:= \left\{ (u,v) \in T_{a_1} \times T_{a_2} : G(u,v) = 0 \right\},
\]
from \eqref{def constainA}, and its radial subset
\begin{equation}\label{def constain}
V_{\rad}:= \left\{ (u,v) \in S_{a_1} \times S_{a_2} : G(u,v) = 0 \right\},
\end{equation}
where
\[
G(u,v) = \int_{\R^3} \left( |\nabla u|^2 + |\nabla v|^2\right) - \frac{3}{4} \int_{\R^3} \left( \mu_1 u^4 + 2\beta u^2 v^2 +\mu_2 v^4 \right).
\]

\begin{lemma}\label{lem: any solution in pohozaev}
If $(u,v)$ is a solution of \eqref{system}-\eqref{normalization} for some $\lambda_1,\lambda_2 \in \R$, then $(u,v) \in V$.
\end{lemma}

\begin{proof}
The Pohozaev identity for system \eqref{system} reads
\begin{equation}\label{Pohozaev classical}
\frac12\int_{\R^3} |\nabla u|^2 + |\nabla v|^2
 = \int_{\R^3} \frac32\left(\lambda_1 u^2 + \lambda_2 v^2\right) + \frac34 \left(\mu_1 u^4
    + 2 \beta u^2 v^2 + \mu_2 v^4\right).
\end{equation}
On the other hand, testing \eqref{system} with $(u,v)$, we find
\begin{align*}
\lambda_1 \int_{\R^3} u^2 & =  \int_{\R^3} |\nabla u|^2 - \int_{\R^3}  \left(\mu_1 u^4 + \beta u^2 v^2\right) \\
\lambda_2 \int_{\R^3} v^2 & =  \int_{\R^3} |\nabla v|^2 - \int_{\R^3}  \left( \beta u^2 v^2 + \mu_2 v^4\right)
\end{align*}
which substituted into \eqref{Pohozaev classical} give the desired result.
\end{proof}

For $(u,v) \in T_{a_1} \times T_{a_2}$, let us set
\[
\Psi_{(u,v)}(s) := J(s \star (u,v)),
\]
where as before $s \star (u,v) = (s \star u, s \star v)$ for short, and $s \star u$ is defined in \eqref{def star}.

It is not difficult to check that $V,V_{\rad}$ are not empty. Actually, directly from the definition, one has much more.

\begin{lemma}\label{lem: expression s bar}
For every $(u,v) \in T_{a_1} \times T_{a_2}$, there exists a unique $s_{(u,v)} \in \R$ such that $(s_{(u,v)} \star (u,v)) \in V$. Moreover, $s_{(u,v)}$ is the unique critical point of $\Psi_{(u,v)}$, which is a strict maximum.
\end{lemma}


\begin{lemma}\label{lem : equality in large space}
There holds
$$ \inf_V J = \inf_{V_{\rad}}J.$$
\end{lemma}

\begin{proof}
In order to prove the lemma we assume by contradiction that there exists $(u,v) \in V$ such that
\begin{equation}\label{star}
0 < J(u,v) < \inf_{V_{\rad}}J.
\end{equation}
For $u \in H^1(\R^3)$ let $u^*$ denotes its Schwarz spherical rearrangement. By the properties of Schwarz symmetrization we have that $J(u^*,v^*) \leq J(u,v)$ and $G(u^*,v^*) \leq G(u,v)=0$. Thus there exists $s_0 \leq 0$ such that $G(s_0 \star (u^*,v^*))=0.$ We claim that
$$J(s_0 \star (u^*,v^*)) \leq e^{2 s_0} J(u^*,v^*).$$
Indeed 
using that $G(s_0 \star (u^*,v^*)) = G(u,v)=0$ we have
\begin{equation}\label{star2}
\begin{aligned}
J(s_0 \star (u^*,v^*)) &= \frac{e^{2s_0}}{6}\int_{\R^3} |\nabla u^*|^2 + |\nabla v^*|^2 \\
&\leq \frac{e^{2s_0}}{6}\int_{\R^3} |\nabla u|^2 + |\nabla v|^2 = e^{2s_0}J(u,v).
\end{aligned}
\end{equation}
Thus
$$0 < J(u,v) < \inf_{V_{\rad}}J \leq J(s_0 \star (u^*,v^*)) \leq e^{2s_0}J(u,v)$$
which contradicts $s_0\le0$. 
\end{proof}

We are ready to complete the proof of Theorem \ref{thm: existence for beta large}.

\begin{proof}[Conclusion of the proof of Theorem \ref{thm: existence for beta large}]
Recalling that any solution of \eqref{system}-\eqref{normalization} stays in $V$, if we have
\begin{equation}\label{above-bound}
 J(\bar u,\bar v)= d  \le \inf\{J(u,v): (u,v) \in V_{\rad}\}
\end{equation}
the thesis follows in view of Lemma \ref{lem : equality in large space}. In order to prove \eqref{above-bound} we choose an arbitrary $(u,v) \in V_{\rad}$ and show that $J(u,v) \ge d$.  At first, since $(u,v) \in V_{\rad}$ implies $(|u|,|v|) \in V_{\rad}$ and $J(u,v) = J(|u|,|v|)$, it is not restrictive to suppose that $u,v \ge 0$ a.e. in $\R^3$. Let us consider the function $\Psi_{(u,v)}$. By Lemma \ref{lem: easy computation} there exists $s_0 \gg 1$ such that $(-s_0) \star (u,v) \in A^+$ and $s_0 \star (u,v) \in C^+$. Therefore, the continuous path
\[
\gamma(t):= ((2t-1)s_0) \star (u,v) \qquad t \in [0,1]
\]
connects $A^+$ with $C^+$, and by Lemmas \ref{lem: connectedness} and \ref{lem: expression s bar} we infer that
\[
d \le \max_{t \in [0,1]} J(\gamma(t)) = J(u,v).
\]
Since this holds for all the elements in $V_{\rad}$, equality \eqref{above-bound} follows.
\end{proof}


\section{Systems with many components}

In this section we prove Theorem \ref{thm: large couplings many comp}. The problem under investigation is \eqref{system many}-\eqref{norm many}: we search for solutions to
\[
\begin{cases}
-\Delta u_i -\lambda_i u_i = \sum_{j=1}^k \beta_{ij} u_j^2 u_i & \text{in $\R^3$} \\
u_i \in H^1(\R^3)
\end{cases} \qquad i=1,\dots,k,
\]
satisfying
\[
\int_{\R^3} u_i^2 = a_i^2 \qquad i=1,\dots,k.
\]
Dealing with multi-components systems, we adopt the notation $\mf{u}:=(u_1,\dots,u_k)$. The first part of the proof is similar to the one of Theorem \ref{thm: existence for beta large}, therefore, we only sketch it.

For $\mf{u} \in S_{a_1} \times \cdots \times S_{a_k}$ (recall definition \eqref{sphere radial}) and $s \in \R$, we consider
\[
J(s \star \mf{u} )
 = \frac{e^{2s}}{2} \int_{\R^3} \sum_i |\nabla u_i|^2
    - \frac{e^{3s}}{4}\int_{\R^3} \sum_{i,j} \beta_{ij} u_i^2 u_j^2
\]
and
\begin{equation}\label{star3}
G(\mf{u}) = \int_{\R^3} \sum_{i=1}^k |\nabla u_i|^2 -\frac{3}{4} \int_{\R^3} \sum_{i,j=1}^k \beta_{ij} u_i^2 u_j^2.
\end{equation}
It is not difficult to extend Lemma \ref{lem: minimax inequality 2} for $k>2$, proving the following:
\begin{lemma}\label{lem: minimax inequality many}
There exists $K>0$ small enough such that
\[
0<\sup_A J< \inf_B J \quad \text{and} \quad  G(u),J(u)>0 \quad \forall u \in A,
\]
where
\[
A := \left\{\mf{u}  \in S_{a_1} \times S_{a_k}: \int_{\R^3} \sum_{i=1}^k |\nabla u_i|^2 \le K \right\} 
\]
and
\[
B:= \left\{\mf{u}  \in S_{a_1} \times S_{a_k}: \int_{\R^3} \sum_{i=1}^k |\nabla u_i|^2 =2K  \right\}.
\]
\end{lemma}

We also introduce the set
\begin{equation}\label{def C many}
C:= \left\{\mf{u}  \in S_{a_1} \times \cdots \times  S_{a_k}: \int_{\R^3} \sum_{i=1}^k |\nabla u_i|^2 \ge 3K \text{ and } J(\mf{u} ) \le 0 \right\},
\end{equation}
and we recall the definition of $w_{a,\mu}$, given in Proposition \ref{prop: explicit}. It is clear that there exist $s_1<0$ and $s_2>0$ such that
\[
s_1 \star (w_{a_1,\beta_{11}},\dots,w_{a_k,\beta_{kk}})=:\hat{\mf{u}} \in A
\]
and
\[
s_2 \star (w_{a_1,\beta_{11}},\dots,w_{a_k,\beta_{kk}})=:\tilde{\mf{u}} \in C.
\]
Setting
\[
\Gamma:= \left\{ \gamma \in \mathcal{C}([0,1],S_{a_1} \times \cdots \times S_{a_k}) :
\gamma(0) = \hat{\mf{u} }, \
\gamma(1) = \tilde{\mf{u}} \right\},
\]
by Lemma \ref{lem: minimax inequality many}, it is possible to argue as in Lemma \ref{lem: appl minimax 2}, showing that there exists a Palais-Smale sequence $(\mf{u}_n)$ for $J$ at level
\[
d:= \inf_{\gamma \in \Gamma} \max_{ t \in [0,1]} J(\gamma(t)),
\]
satisfying the additional condition
\begin{equation}\label{weak pohozaev many}
G(\mf{u})=\int_{\R^3} \sum_{i=1}^k |\nabla u_i|^2 - \frac{3}{4} \int_{\R^3} \sum_{i,j=1}^k \beta_{ij} u_i^2 u_j^2 = o(1),
\end{equation}
with $o(1) \to 0$ as $n \to \infty$. Moreover $u_{i,n}^- \to 0$ a.e. in $\R^3$ as $n \to \infty$, for any $i$. Notice that the value $d$ depends on all the masses $a_i$ and on all the couplings $\beta_{ij}$.

It remains to show that $\mf{u}_n \to \bar{\mf{u}}$ strongly in $H^1,$ and the limit is a solution of \eqref{system many}-\eqref{norm many}. In order to do this, we argue as for the $2$-components system: thanks to \eqref{weak pohozaev many}, up to a subsequence $\mf{u}_n \to \bar{\mf{u}}$ weakly in $H^1(\R^3,\R^k)$, strongly in $L^4(\R^3,\R^k)$, a.e. in $\R^3$. As before we arrive at the conclusion that $\bar{\mf{u}}$ is a solution of \eqref{system many} for some $\lambda_1,\dots,\lambda_k \in \R$. We can also suppose that one of these parameters, say $\lambda_1$, is strictly negative. Thus, Lemma \ref{lem: strong convergence} implies that $u_{1,n} \to \bar u_1$ strongly in $H^1(\R^3)$. If by contradiction $u_{j,n} \not \to \bar u_j$ strongly in $H^1(\R^3)$ for some $j$, then $\lambda_j \ge 0$, and by Lemma A.2 in \cite{Ikoma} we deduce that $\bar u_j \equiv 0$. To complete the proof, we aim at showing that $\bar u_i \not \equiv 0$ for every $i$, and to do this it is necessary to substantially modify the argument used for Theorem \ref{thm: existence for beta small}.

We divide the set of indexes $\{1,\dots,k\}$ in two subsets:
\[
\mathcal{I}_1:=\{i \in \{1,\dots,k\}: \lambda_i<0\} \quad \text{and} \quad  \mathcal{I}_2:=\{i \in \{1,\dots,k\}: \lambda_i \ge 0\}.
\]
Notice that $1 \in \mathcal{I}_1$, so that the cardinality of $\mathcal{I}_2$ is at most $k-1$, and that the absurd assumption can be written as $\mathcal{I}_2 \neq \emptyset$. Up to a relabelling, we can suppose for the sake of simplicity that
\[
\mathcal{I}_1:= \{1,\dots,m\} \quad \text{and} \quad \mathcal{I}_2:=\{m+1,\dots,k\}
\]
for some $1\le m<k$. By strong convergence (and by the maximum principle)
\[
\begin{cases}
-\Delta \bar u_i -\lambda_i \bar u_i = \sum_{j \in \mathcal{I}_1} \beta_{ij} \bar u_i \bar u_j^2 & \text{in $\R^3$} \\
\bar u_i >0 & \text{in $\R^3$} \\
\int_{\R^3} \bar u_i^2 = a_i^2,
\end{cases} \qquad \forall i \in \mathcal{I}_1,
\]
while $\bar u_i \equiv 0$ for every $i \in \mathcal{I}_2$. As in Lemma \ref{lem: any solution in pohozaev}, this implies that $(\bar u_{1},\dots,\bar u_{m}) \in V_{\rad}^{\mathcal{I}_1}$, where
\[
V_{\rad}^{\mathcal{I}_1}:= \left\{\mf{u} \in S_{a_{1}} \times \cdots \times S_{a_{m}}: \int_{\R^3} \sum_{i=1}^m |\nabla u_i|^2 = \frac{3}{4} \int_{\R^3} \sum_{i,j=1}^m \beta_{i j} u_i^2 u_j^2 \right\}.
\]
Therefore
\begin{equation}\label{J di uguale a}
J(\bar{\mf{u}}) = J(\bar u_{1},\dots,\bar u_{m},0,\dots,0) \ge \inf_{V_{\rad}^{\mathcal{I}_1}} J.
\end{equation}
Notice that in the last term we used $J$ to denote the functional associated to a $m$ components system, while in the previous terms $J$ is used for the functional associated to the full $k$ components system. This should not be a source of misunderstanding.

The value $J(\mf{u})$ can also be characterized in a different way: by \eqref{weak pohozaev many}, strong $L^4$-convergence, and recalling that $(\bar u_{1},\dots,\bar u_{m}) \in V_{\rad}^{\mathcal{I}_1}$, we have also
\begin{equation}\label{c uguale a J di}
\begin{split}
d & =\lim_{n \to \infty} J(\mf{u}_n) = \lim_{n \to \infty} \frac{1}{8}\int_{\R^3} \sum_{i,j=1}^k \beta_{ij} u_{i,n}^2 u_{j,n}^2 \\
& = \frac{1}{8}\int_{\R^3} \sum_{i,j=1}^m \beta_{ij} \bar u_{i}^2 \bar u_{j}^2 = J(\bar u_1,\dots \bar u_m,0,\dots,0) = J(\bar{\mf{u}}).
\end{split}
\end{equation}
A comparison between \eqref{J di uguale a} and \eqref{c uguale a J di} reveals that
\begin{equation}\label{eq from which abs}
d \ge \inf_{V_{\rad}^{\mathcal{I}_1}} J.
\end{equation}
To find a contradiction, we shall provide an estimate from above on $d$, an estimate from below on $\inf_{V_{\rad}^{\mathcal{I}_1}} J$, and show that these are not compatible with \eqref{eq from which abs}.

\medskip

\noindent \paragraph{\textbf{Upper estimate on $d$.}} First of all, we state the extension of Lemma \ref{lem: expression s bar}.

\begin{lemma}\label{lem: expression s bar many}
For every $\mf{u} \in S_{a_1} \times \cdots \times S_{a_k}$, there exists a unique $s_{\mf{u}} \in \R$ such that $(s_{\mf{u}} \star (u,v)) \in V_{\rad}$. Moreover, $s_{\mf{u}}$ is the unique critical point of $\Psi_{\mf{u}}(s):=J(s \star \mf{u})$, which is a strict maximum.
\end{lemma}

We shall now prove two variational characterizations for $d$.

\begin{lemma}\label{lem: var char 1}
It results that
\[
d= \inf_{V_{\rad}} J,
\]
where
\[
V_{\rad} := \left\{\mf{u} \in S_{a_1} \times \cdots \times S_{a_k}:
\int_{\R^3} \sum_{i=1}^k |\nabla u_i|^2 = \frac{3}{4} \int_{\R^3} \sum_{i,j=1}^k \beta_{i j} u_i^2 u_j^2 \right\}.
\]
\end{lemma}

\begin{proof}
This can be done arguing as in the proof of Theorem \ref{thm: existence for beta large}. Firstly, we introduce the sets
\[
A^+:= \left\{\mf{u} \in A: \text{$u_i \ge 0$ a.e. in $\R^3$ for every $i$}\right\} 
\]
and
\[
C^+ := \left\{\mf{u} \in C: \text{$u_i \ge 0$ a.e. in $\R^3$ for every $i$}\right\},
\]
where $A$ and $C$ have been defined in Lemma \ref{lem: minimax inequality many} and \eqref{def C many}, respectively. Slightly modifying the proof of Lemma \ref{lem: connectedness}, one can check that $A^+$ and $C^+$ are connected by arcs, so that for any $\mf{u} \in A^+$ and $\mf{u}' \in C^+$ it results that
\[
d= \inf_{\gamma \in \Gamma(\mf{u}, \mf{u}')} \max_{t \in [0,1]} J(\gamma(t)),
\]
where
\[
\Gamma(\mf{u},\mf{u}'):= \left\{ \gamma \in \mathcal{C}([0,1],S_{a_1} \times \cdots \times S_{a_k}): \gamma(0) =\mf{u}, \  \gamma(1) =\mf{u}'  \right\}.
\]
As in the conclusion of the proof of Theorem \ref{thm: existence for beta large}, from this it follows that
\begin{equation}\label{101}
d \le \inf_{V_{\rad}} J.
\end{equation}
We have to check that also the reverse inequality holds. To this aim, we claim that
\begin{equation}\label{claim: on intersection}
\text{for any path $\gamma$ from $A^+$ to $C^+$ there exists $t \in (0,1)$ such that $\gamma(t) \in V_{\rad}$}.
\end{equation}
Once that the claim is proved, it is possible to observe that for any such $\gamma$
\[
\inf_{V_{\rad}} J \le J(\gamma(t)) \le \max_{t \in [0,1]} J(\gamma(t)),
\]
and taking the infimum over all the admissible $\gamma$ from $A^+$ to $C^+$, we deduce that
\[
\inf_{V_{\rad}} J \le d,
\]
which together with \eqref{101} completes the proof. Thus, it remains only to verify claim \eqref{claim: on intersection}. Notice that
\[
\mf{u}\in V_{\rad} \quad \Longleftrightarrow \quad G(\mf{u}) =0,
\]
where $G$ has been defined in \eqref{star3}. There we showed that $G(\mf{u})>0$ for every $\mf{u} \in A$. Moreover, $J(\mf{u}) \le 0$ for every $\mf{u} \in C^+$, which directly implies
\[
G(\mf{u}) \le -\frac{1}{4} \int_{\R^3} \sum_{i,j} \beta_{ij}  u_i^2 u_j^2 <0
 \qquad \text{for all } \mf{u} \in C.
\]
Thus, by continuity, for any $\mf{u} \in A^+$, any $\mf{u}' \in C^+$, and any $\gamma \in \Gamma(\mf{u},\mf{u}')$, there exists $t \in (0,1)$ such that $G(\gamma(t))=0$, which proves the claim.
\end{proof}

We introduce a Rayleigh-type quotient as
\[
\mathcal{R}(\mf{u}):=  \frac{8 \left( \int_{\R^3} \sum_{i=1}^k |\nabla u_i|^2 \right)^3  }{27 \left(\int_{\R^3} \sum_{i,j=1}^k \beta_{ij} u_i^2 u_j^2 \right)^2}.
\]

\begin{lemma}\label{lem: rayleigh}
There holds that
\[
d = \inf_{V_{\rad}} J =\inf_{S_{a_1} \times \cdots \times S_{a_k}} \mathcal{R}.
\]
\end{lemma}
\begin{proof}
If $\mf{u} \in V_{\rad}$, then
\[
\frac{4\int_{\R^3} \sum_{i=1}^k |\nabla u_i|^2 }{3 \int_{\R^3} \sum_{i,j} \beta_{ij} u_i^2 u_j^2} =1 \quad \text{and} \quad J(u_1,\dots,u_k) = \frac{1}{6} \int_{\R^3} \sum_{i=1}^k |\nabla u_i|^2.
\]
Therefore
\[
J(\mf{u}) = \frac{1}{6} \int_{\R^3} \sum_{i=1}^k |\nabla u_i|^2 \cdot \left( \frac{4\int_{\R^3} \sum_{i=1}^k |\nabla u_i|^2 }{3 \int_{\R^3} \sum_{i,j} \beta_{ij} u_i^2 u_j^2} \right)^2 = \mathcal{R}(\mf{u}),
\]
which proves that $\inf_{V_{\rad}} J \ge \inf_{S_{a_1} \times \cdots \times S_{a_k}} \mathcal{R}$. On the other hand, it is not difficult to check that
\[
\mathcal{R}(s \star \mf{u}) = \mathcal{R}(\mf{u}) \qquad \text{for all }s \in \R, \ \mf{u} \in S_{a_1} \times \cdots \times S_{a_k}.
\]
By Lemma \ref{lem: expression s bar many}, we conclude that
\[
\mathcal{R}(\mf{u}) = \mathcal{R}(s_{\mf{u}} \star \mf{u}) = J(s_{\mf{u}} \star \mf{u}) \ge \inf_{V_{\rad}} J
\]
for every $\mf{u}\in S_{a_1} \times \cdots \times S_{a_k}$.
\end{proof}

The previous characterization makes it possible to derive an upper bound on $d$.

\begin{lemma}\label{lem: upper bound on c many}
With $C_0$ and $C_1$ defined in \eqref{def of constants}, there holds
\[
d \le \frac{C_0C_1 \left( \sum_i a_i^2\right)^3   }{8 \left( \sum_{i,j} \beta_{ij} a_i^2 a_j^2\right)^2 }
\].
\end{lemma}
\begin{proof}
By Lemma \ref{lem: rayleigh}, we have
\[
d \le \mathcal{R}\left(w_{a_1,C_0/a_1^2}, \dots, w_{a_k,C_0/a_k^2}\right),
\]
where we recall that $w_{a,\mu}$ has been defined in Proposition \ref{prop: explicit}. Using the explicit expression of $w_{a_i,C_0/a_i^2}$, we can compute
\[
\int_{\R^3} w_{a_i,C_0/a_i^2}^2 w_{a_j,C_0/a_j^2}^2 = \frac{a_i^2 a_j^2}{C_0^2} \int_{\R^3} w_0^4 = \frac{a_i^2 a_j^2 C_1}{C_0^2}.
\]
Recalling also \eqref{int grad} and \eqref{int fourth}, we deduce that
\[
\mathcal{R}\left(w_{a_1,C_0/a_1^2}, \dots, w_{a_k,C_0/a_k^2}\right) = \frac{\displaystyle \left( \sum_i \frac{C_1a_i^2}{C_0}\right)^3   }{ \displaystyle 8\left(\sum_{i,j}  \beta_{ij} \frac{a_i^2 a_j^2 C_1}{C_0^2}\right)^2} = \frac{C_0C_1 \left( \sum_i a_i^2\right)^3   }{8 \left( \sum_{i,j} \beta_{ij} a_i^2 a_j^2\right)^2 },
\]
and the lemma follows.
\end{proof}
\medskip

\noindent \paragraph{\textbf{Lower estimate for $\inf_{V_{\rad}^{\mathcal{I}_1}} J$.}}

Recall that we supposed, for the sake of simplicity, that $\mathcal{I}_1=\{1,\dots,m\}$ for some $1 \le m<k$. Let us introduce
\[
\mathcal{R}_{\mathcal{I}_1}(u_1,\dots,u_m):= \frac{8 \left( \int_{\R^3} \sum_{i=1}^m |\nabla u_i|^2 \right)^3  }{27 \left(\int_{\R^3} \sum_{i,j=1}^m \beta_{ij} u_i^2 u_j^2 \right)^2}.
\]
Exactly as in Lemma \ref{lem: rayleigh}, one can prove that
\begin{equation}\label{inf as Ray 2}
\inf_{V_{\rad}^{\mathcal{I}_1}} J = \inf_{S_{a_1} \times \cdots \times S_{a_m}} \mathcal{R}_{\mathcal{I}_1}.
\end{equation}

\begin{lemma}\label{lem: lower on c many}
\[
\inf_{V_{\rad}^{\mathcal{I}_1}} J  \ge \frac{C_0 C_1}{\displaystyle 8 \left[ \max_{1 \le j \le m} \{\beta_{jj} a_j\} + \frac{m-1}{m}\max_{1 \le i \neq j \le m} \{  \beta_{ij} a_i^{1/2} a_j^{1/2}\}\right]^2}.
\]
\end{lemma}
\begin{proof}
We recall that
\[
\sum_{1 \le i \neq j \le m} x_i x_j \le \frac{m-1}{m} \left(\sum_{i=1}^m x_i\right)^2 
\qquad \text{for all }m \in \N, \, x_1,\dots,x_m>0.
\]
Thus, by the Young's and the Gagliardo-Nirenberg's inequalities we have
\begin{align*}
\sum_{i,j =1}^m \int_{\R^3} & \beta_{ij} u_i^2 u_j^2  \le \sum_{i,j =1}^m \frac{\beta_{ij}}{2} \left(\int_{\R^3}  u_i^4\right)^{\frac12}\left( \int_{\R^3} u_j^4\right)^{\frac12}  \\
& \le S \sum_{i,j=1}^m \beta_{ij} \sqrt{a_i a_j} \left(\int_{\R^3}  |\nabla u_i|^2\right)^{\frac34}\left( \int_{\R^3} |\nabla u_j|^2\right)^{\frac34}  \\
& \le S \left[ \max_{1 \le j \le m} \{\beta_{jj} a_j \} \sum_{i=1}^m \left(\int_{\R^3}  |\nabla u_i|^2\right)^{\frac32} \right. \\
& \qquad \quad \left. + \max_{1 \le i \neq j \le m} \{\beta_{ij} \sqrt{a_i a_j}\} \sum_{i \neq j} \left(\int_{\R^3}  |\nabla u_i|^2\right)^{\frac34}\left( \int_{\R^3} |\nabla u_j|^2\right)^{\frac34} \right] \\
& \le S\left[ \max_{1 \le j \le m} \{\beta_{jj} a_j\} + \frac{m-1}{m}\max_{1 \le i \neq j \le m} \{  \beta_{ij} \sqrt{a_i a_j}\} \right] \left( \sum_{i=1}^m \int_{\R^3} |\nabla u_i|^2 \right)^{\frac32}
%
%
\end{align*}
for every $(u_1,\dots,u_m) \in S_{a_1} \times \cdots \times S_{a_m}$. Thanks to the characterization of $S$ in terms of $C_0$ and $C_1$, Lemma \ref{lem: value S}, and the definition of $\mathcal{R}_{\mathcal{I}_1}$, we deduce that
\[
\inf_{S_{a_1} \times \cdots \times S_{a_m}} \mathcal{R}_{\mathcal{I}_1} \ge \frac{C_0 C_1}{\displaystyle 8 \left[ \max_{1 \le j \le m} \{\beta_{jj} a_j\} + \frac{m-1}{m}\max_{1 \le i \neq j \le m} \{  \beta_{ij} a_i^{1/2} a_j^{1/2}\}\right]^2},
\]
and the desired result follows (recall also \eqref{inf as Ray 2}).
\end{proof}

Before proceeding with the conclusion of the proof, we observe that, as for systems with $2$ components, the ground state radial level coincides with the ground state level in the all space, in the following sense.

\begin{lemma}\label{lem: identical on all space}
There holds that
$$\inf_V J = \inf_{V_{\rad}}J$$
where
\begin{equation}\label{def constainAA}
V:= \left\{ \mf{u} \in T_{a_1} \times \cdots \times T_{a_k} : G(\mf{u}) = 0 \right\},
\end{equation}
\end{lemma}
\begin{proof}
The proof relies on Lemma \ref{lem: expression s bar many}, which also holds when $\mf{u} \in T_{a_1} \times \cdots \times T_{a_k}$, and follows the line of the proof of Lemma \ref{lem : equality in large space}.
\end{proof}

\begin{proof}[Conclusion of the proof of Theorem \ref{thm: large couplings many comp}]
We want to show that, under assumption \eqref{condition couplings}, inequality \eqref{eq from which abs} cannot be satisfied. If
\begin{equation}\label{103}
\frac{C_0 C_1 \left( \sum_{i=1}^k a_i^2\right)^3   }{8 \displaystyle \left( \sum_{i,j=1}^k \beta_{ij} a_i^2 a_j^2\right)^2 } < \frac{C_0 C_1}{8 \displaystyle \left[ \max_{1 \le j \le m} \{\beta_{jj} a_j\} + \frac{m-1}{m}\max_{1 \le i \neq j \le m} \{  \beta_{ij} a_i^{1/2} a_j^{1/2}\}\right]^2}
\end{equation}
then by Lemmas \ref{lem: upper bound on c many}, \ref{lem: lower on c many} and \ref{lem: identical on all space} the theorem follows. A condition which implies the validity of \eqref{103}, and hence of the theorem, is assumption \eqref{condition couplings}.
\end{proof}

\begin{remark}\label{rem: diff many 2}
We emphasize the main difference between the concluding arguments in Theorem \ref{thm: existence for beta large} and \ref{thm: large couplings many comp}. In the former case to obtain a contradiction one has to compare the value $d$  with two fixed quantities $\ell(a_1,\mu_1)$ and $\ell(a_2,\mu_2)$, which do not depend on $\beta$; on one side, argueing by contradiction one has $d= \ell(a_i,\mu_i)$ for some $i$; on the other hand, we have seen that it is sufficient to take $\beta$ very large to have $d<\min\{\ell(a_1,\mu_1),\ell(a_2,\mu_2)\}$, which gives a contradiction. For systems with many components the situation is much more involved: the crucial equality for Theorem \ref{thm: large couplings many comp} is \eqref{eq from which abs}, which involves two quantities \emph{both depending} on the coupling parameters. It would be tenting to think that the natural assumption in Theorem \ref{thm: large couplings many comp} is $\beta_{ij} \ge \bar \beta$ for every $i \neq j$. But if we make some $\beta_{ij}$ too large, than both the sides in \eqref{eq from which abs} becomes very small, and without any condition on the other parameters ($\beta_{ij}$ and $a_i$) it seems hard to obtain a contradiction. Notice also that we do not have any control on the set $\mathcal{I}_1$, which makes the problem even more involved and imposes an assumption involving all the possible choices of $\mathcal{I}_1$.

For all these reasons we think that condition \eqref{condition couplings}, which seems strange at a first glance, is not so unnatural.
\end{remark}

\section{Orbital stability}

This section is devoted to the proof of Theorem \ref{thm: instability}, and we focus on a general $k$ components system. Let $(\bar \lambda_1,\dots,\bar \lambda_k,\bar u_1,\dots,\bar u_k)$ be the solution of \eqref{system many} found in Theorem \ref{thm: large couplings many comp}. The crucial fact is that, by Lemma \ref{lem: identical on all space}, $J(\bar{\mf{u}}) = \inf_{V} J$, where we recall that
$V := \{\mf{u} \in T_{a_1} \times \cdots \times T_{a_k} : G(\mf{u})=0\}$, with
\[
G(\mf{u}) = \int_{\R^3} \sum_{i=1}^k |\nabla u_i|^2
             - \frac34 \int_{\R^3} \sum_{i,j=1}^k \beta_{ij} u_i^2 u_j^2.
\]
The dynamics of \eqref{syst schrod many} takes place in $H^1(\R^3, \CC^k)$. By using similar arguments as in the proof of Lemmas \ref{lem : equality in large space} and \ref{lem: identical on all space}, with $(u^*,v^*)$ replaced by $(|u|,|v|)$, one can show that
$$\inf_{V_{\CC}}  J= \inf_{V}J$$
where
$$V_{\CC}:=\left\{ \mf{u} \in T_{a_1}^{\CC} \times \cdots \times T_{a_k}^{\CC} : G(\mf{u})=0\right\},$$
and
$$T_a^{\CC}:= \left\{ u \in H^1(\R^3, \CC^k) : \int_{\R^3}|u|^2 = a^2\right\}.$$
Let us introduce the function
\[
g_{\mf{u}}(t):= \frac{t^2}{2}\int_{\R^3} \sum_{i=1}^k |\nabla u_i|^2 - \frac{t^3}{4} \int_{\R^3} \sum_{i,j=1}^k \beta_{ij} |u_i|^2 |u_j|^2,
\]
defined for $t>0$. Notice that $g_{\mf{u}}(t)  = J(\log t \star \mf{u})$. It is clear that for any $\mf{u} \in H^1(\R^3,\CC^k)$ there exists a unique critical point $t_{\mf{u}}>0$ for $g_{\mf{u}}$, which is a strict maximum, and that $\log t_{\mf{u}} \star \mf{u} \in V_{\CC}$. Moreover, the function $g_{\mf{u}}$ is concave in $(t_{\mf{u}} ,+\infty)$.

\begin{lemma}\label{lem: lem for instability}
Let $d : = \inf\{J(\mf{u}): \mf{u} \in V_{\CC}\}$. Then
\[
G(\mf{u})<0 \quad \Longrightarrow \quad  G(\mf{u}) \le J(\mf{u})-d.
\]
\end{lemma}
\begin{proof}
By a direct computation $G(\mf{u}) = g_{\mf{u}}'(1)$. Thus, the condition $G(\mf{u})<0$ implies that $t_{\mf{u}}<1$, and $g_{\mf{u}}$ is concave in $(t_{\mf{u}},+\infty)$. As a consequence,
\[
g_{\mf{u}}(1) \ge g_{\mf{u}}(t_{\mf{u}}) +(1-t_{\mf{u}}) g_{\mf{u}}'(1)  \ge g_{\mf{u}}(t_{\mf{u}}) + G(\mf{u})  \ge d+ G(\mf{u}),
\]
and since $g_{\mf{u}}(1)  = J(\mf{u})$ the thesis follows.
\end{proof}

\begin{proof}[Conclusion of the proof of Theorem \ref{thm: large couplings many comp}]
Let $\mf{u}_s:= s \star \bar{\mf{u}}$. Since $\bar{\mf{u}} \in  V_{\CC}$, it follows that $G(\mf{u}_s) <0$ for every $s>0$. Let $\Phi^s=(\Phi_1^s,\dots,\Phi_k^s)$ be the solution of system \eqref{syst schrod many} with initial datum $\mf{u}_s$, defined on the maximal interval $(T_{\min},T_{\max})$. By continuity, provided $|t|$ is sufficiently small we have $G(\Phi^s(t)) < 0$.
Therefore, by Lemma \ref{lem: lem for instability} and recalling that the energy is conserved along trajectories of \eqref{syst schrod many}, we have
\[
G(\Phi^s(t)) \le J(\Phi^s(t))-d =  J(\mf{u}_s)-d =:-\delta <0
\]
for any such $t$, and by continuity again we infer that $G(\Phi^s(t)) \le -\delta$ for every $t \in (T_{\min},T_{\max})$. To obtain a contradiction, we recall that the virial identity (see Proposition 6.5.1 in \cite{Caz} for the identity associated to the scalar equation; dealing with a gradient-type system, the computations are very similar) establishes that
\[
f''_s(t) = 8 G(\Phi^s(t)) \le - 8 \delta<0 \quad \text{for} \quad f_s(t):= \int_{\R^3} |x|^2 \sum_{i=1}^k |\Phi^s_i(t,x)|^2\,dx
\]
and as a consequence
\[
0 \le f_s(t) \le -\delta t^2 + O(t) \qquad \text{for all }t \in (-T_{\min},T_{\max}).
\]
Since the right hand side becomes negative for $|t|$ sufficiently large, it is necessary that both $T_{\min}$ and $T_{\max}$ are bounded. This proves that, for a sequence of initial data arbitrarily close to $\bar{\mf{u}}$, we have blow-up in finite time, implying orbital instability.
\end{proof}




\begin{thebibliography}{10}

\bibitem{AmbrosettiColorado}
A.~Ambrosetti and E.~Colorado.
\newblock Standing waves of some coupled nonlinear {S}chr{\"o}dinger equations.
\newblock {\em J. Lond. Math. Soc. (2)} 75(1), 67--82, 2007.

\bibitem{Bartsch}
T.~Bartsch.
\newblock Bifurcation in a multicomponent system of nonlinear {S}chr\"odinger
  equations.
\newblock {\em J. Fixed Point Theory Appl.} 13(1), 37--50, 2013.

\bibitem{BaDaWa}
T.~Bartsch, N.~Dancer and Z.-Q.~Wang.
\newblock A {L}iouville theorem, a-priori bounds, and bifurcating branches of
  positive solutions for a nonlinear elliptic system.
\newblock {\em Calc. Var. Partial Differential Equations} 37(3-4), 345--361,
  2010.

\bibitem{BaVa}
T.~Bartsch and S.~De Valeriola
\newblock  Normalized solutions of nonlinear Schr\"odinger equations.
\newblock {\em Arch. Math.} 100(1), 75--83, 2013.

\bibitem{BarJea}
T.~Bartsch and L.~Jeanjean.
\newblock Normalized solutions for nonlinear Schr\"odinger systems.
\newblock preprint.

\bibitem{BaWa}
T.~Bartsch and Z.-Q. Wang.
\newblock Note on ground states of nonlinear {S}chr\"odinger systems.
\newblock {\em J. Partial Differential Equations} 19(3), 200--207, 2006.

\bibitem{BaWaWei}
T.~Bartsch, Z.-Q. Wang and J. Wei.
\newblock Bound states for a coupled Schrödinger system.
\newblock {\em Journal of Fixed Point Theory and Applications} 2, 353-367, 2007.

\bibitem{BeJeLu}
J.~Bellazzini, L.~Jeanjean and T-J.~Luo.
\newblock Existence and instability of standing waves with prescribed norm for a class of Schr\"odinger-Poisson equations.
\newblock {\em Proc. London Math. Soc.} 107(3), 303--339, 2013.

\bibitem{BerCaz}
H.~Berestycki and T.~Cazenave.
\newblock Instabilit\'e des \'etats stationnaires dans les \'equations de
  {S}chr\"odinger et de {K}lein-{G}ordon non lin\'eaires.
\newblock {\em C. R. Acad. Sci. Paris S\'er. I Math.} 293(9), 489--492, 1981.

\bibitem{Caz}
T.~Cazenave.
\newblock {\em Semilinear {S}chr\"odinger equations}, volume~10 of {\em Courant
  Lecture Notes in Mathematics}.
\newblock New York University, Courant Institute of Mathematical Sciences, New
  York; American Mathematical Society, Providence, RI, 2003.

\bibitem{ChZo}
Z.~Chen and W.~Zou.
\newblock An optimal constant for the existence of least energy solutions of a
  coupled {S}chr\"odinger system.
\newblock {\em Calc. Var. Partial Differential Equations} 48(3-4), 695--711,
  2013.

\bibitem{Cor}
S.~Correia.
\newblock Stability of ground states for a system of $m$ coupled semilinear
  schr\"odinger equations.
\newblock preprint 2015, http://arxiv.org/abs/1502.07913.

\bibitem{esry-etal:1997}
B. D.~Esry, C. H.~Greene, J. P.~Burke Jr. and J. L.~Bohn.
\newblock Hartree-Fock theory for double condensates.
\newblock {\em  Phys.\ Rev.\ Lett.} 78, 3594, 1997.





\bibitem{Ghoussoub}
N.~Ghoussoub.
\newblock {\em Duality and perturbation methods in critical point theory},
  volume 107 of {\em Cambridge Tracts in Mathematics}.
\newblock Cambridge University Press, Cambridge, 
\newblock With appendices by David Robinson, 1993.

\bibitem{IkoUni}
N.~Ikoma.
\newblock Uniqueness of positive solutions for a nonlinear elliptic system.
\newblock {\em NoDEA Nonlinear Differential Equations Appl.} 16(5), 555--567,
  2009.

\bibitem{Ikoma}
N.~Ikoma.
\newblock Compactness of minimizing sequences in nonlinear {S}chr\"odinger
  systems under multiconstraint conditions.
\newblock {\em Adv. Nonlinear Stud.} 14(1), 115--136, 2014.

\bibitem{Jeanjean}
L.~Jeanjean.
\newblock Existence of solutions with prescribed norm for semilinear elliptic
  equations.
\newblock {\em Nonlinear Anal.} 28(10), 1633--1659, 1997.

\bibitem{JeLuWa}
L.~Jeanjean, T.-J.~Luo and Z.-Q.~Wang.
\newblock Multiple normalized solutions for quasi-linear Schr\"odinger equations.
\newblock {\em J. Differential Equations}. doi:10.1016/j.jde.2015.05.008.

\bibitem{Kwong}
M.~K. Kwong.
\newblock Uniqueness of positive solutions of {$\Delta u-u+u^p=0$} in {${\mathbf
  R}^n$}.
\newblock {\em Arch. Rational Mech. Anal.} 105(3), 243--266, 1989.

\bibitem{LeCoz}
S.~Le~Coz.
\newblock A note on {B}erestycki-{C}azenave's classical instability result for
  nonlinear {S}chr\"odinger equations.
\newblock {\em Adv. Nonlinear Stud.} 8(3), 455--463, 2008.

\bibitem{LinWei}
T.-C. Lin and J.~Wei.
\newblock Ground state of {$N$} coupled nonlinear {S}chr{\"o}dinger equations
  in {$\mathbb{R}^n$}, {$n\leq 3$}.
\newblock {\em Comm. Math. Phys.} 255(3), 629--653, 2005.

\bibitem{LinWei2}
T.-C.~Lin and J.~Wei.
\newblock Spikes in two coupled nonlinear {S}chr\"odinger equations.
\newblock {\em Ann. Inst. H. Poincar\'e Anal. Non Lin\'eaire} 22(4), 403--439,
  2005.

\bibitem{Li1}
P.L~Lions.
\newblock The concentration-compactness principle in the calculus of variation. The locally compact case,
part I.
\newblock {\em Ann. Inst. H. Poincar\'e Anal. Non Lin\'eaire} 1(2), 109--145, 1984.

\bibitem{Li2}
P.L. Lions.
\newblock The concentration-compactness principle in the calculus of variation. The locally compact case,
part II.
\newblock {\em  Ann. Inst. H. Poincar\'e Anal. Non Lin\'eaire} 1(4), 223--283, 1984.

\bibitem{LiuWang}
Z.~Liu and Z.-Q.~Wang.
\newblock Ground states and bound states of a nonlinear {S}chr{\"o}dinger
  system.
\newblock {\em Adv. Nonlinear Stud.} 10(1), 175--193, 2010.

\bibitem{Lu}
T.-J.~Luo
\newblock Multiplicity of normalized solutions for a class of nonlinear Schr\"odinger-Poisson-Slater equations.
\newblock {\em J. Math. Anal. Appl.} 416(1), 195--204, 2014.

\bibitem{MaiaMontefuscoPellacci}
L.~A. Maia, E.~Montefusco, and B.~Pellacci.
\newblock Positive solutions for a weakly coupled nonlinear {S}chr\"odinger
  system.
\newblock {\em J. Differential Equations} 229(2), 743--767, 2006.

\bibitem{MaMoPe2}
L.~d.~A. Maia, E.~Montefusco, and B.~Pellacci.
\newblock Orbital stability property for coupled nonlinear {S}chr\"odinger
  equations.
\newblock {\em Adv. Nonlinear Stud.} 10(3), 681--705, 2010.

\bibitem{malomed:2008}
B. Malomed.
\newblock Multi-component Bose-Einstein condensates: Theory.
In: P.G. Kevrekidis, D.J. Frantzeskakis, R. Carretero-Gonzalez (Eds.),
Emergent Nonlinear Phenomena in Bose-Einstein Condensation,
Springer-Verlag, Berlin, 287-305, 2008.

\bibitem{Mand}
R.~Mandel.
\newblock Minimal energy solutions for cooperative nonlinear {S}chr\"odinger
  systems.
\newblock {\em NoDEA Nonlinear Differential Equations Appl.} 22(2), 239--262,
  2015.

\bibitem{NgWa}
N.~V. Nguyen and Z.-Q. Wang.
\newblock Orbital stability of solitary waves for a nonlinear {S}chr\"odinger
  system.
\newblock {\em Adv. Differential Equations} 16(9-10), 977--1000, 2011.

\bibitem{NoTaTeVeJEMS}
B.~Noris, H.~Tavares, S.~Terracini, and G.~Verzini.
\newblock Convergence of minimax structures and continuation of critical points
  for singularly perturbed systems.
\newblock {\em J. Eur. Math. Soc. (JEMS)} 14(4), 1245--1273, 2012.

\bibitem{NoTaVe1}
B.~Noris, H.~Tavares, and G.~Verzini.
\newblock Existence and orbital stability of the ground states with prescribed mass for the $L^2$-critical and supercritical NLS on bounded domains.
\newblock {\em Anal. PDE} 7(8), 1807--1838, 2014.



\bibitem{NoTaVe2}
B.~Noris, H.~Tavares, and G.~Verzini.
\newblock Stable solitary waves with prescribed $L^2$-mass for the cubic
  Schr\"odinger system with trapping potentials.
\newblock {\em Discrete Contin. Dyn. Syst.-A} 35(12),  6085--6112, 2015.

\bibitem{Ohta}
M.~Ohta.
\newblock Stability of solitary waves for coupled nonlinear {S}chr\"odinger
  equations.
\newblock {\em Nonlinear Anal.} 26(5), 933--939, 1996.

\bibitem{SaWa}
Y.~Sato and Z.-Q. Wang.
\newblock Least energy solutions for nonlinear {S}chr\"odinger systems with
  mixed attractive and repulsive couplings.
\newblock {\em Adv. Nonlinear Stud.} 15(1), 1--22, 2015.

\bibitem{Sirakov}
B.~Sirakov.
\newblock Least energy solitary waves for a system of nonlinear
  {S}chr{\"o}dinger equations in {$\mathbb{R}^n$}.
\newblock {\em Comm. Math. Phys.} 271(1), 199--221, 2007.

\bibitem{Soave}
N.~Soave.
\newblock On existence and phase separation of solitary waves for nonlinear
  {S}chr{\"o}dinger systems modelling simultaneous cooperation and competition.
\newblock {\em Calc. Var. Partial Differential Equations} 53(3-4), 689--718, 2015.

\bibitem{SoaveTavares}
N.~Soave and H.~Tavares.
\newblock New existence and symmetry results for least energy positive
  solutions of Schr\"odinger systems with mixed competition and cooperation
  terms.
\newblock preprint 2014, http://arxiv.org/abs/1412.4336.

\bibitem{TaTe}
H.~Tavares and S.~Terracini.
\newblock Sign-changing solutions of competition-diffusion elliptic systems and
  optimal partition problems.
\newblock {\em Ann. Inst. H. Poincar\'e Anal. Non Lin\'eaire} 29(2), 279--300,
  2012.

\bibitem{TerVer}
S.~Terracini and G.~Verzini.
\newblock Multipulse phases in {$k$}-mixtures of {B}ose-{E}instein condensates.
\newblock {\em Arch. Ration. Mech. Anal.} 194(3), 717--741, 2009.

\bibitem{WeiWeth}
J.~Wei and T.~Weth.
\newblock Radial solutions and phase separation in a system of two coupled
  {S}chr\"odinger equations.
\newblock {\em Arch. Ration. Mech. Anal.} 190(1), 83--106, 2008.

\end{thebibliography}

\end{document}